\documentclass[12pt, a4paper]{elsarticle}
\usepackage{amsthm,amsmath,amssymb,extsizes,easymat,graphicx,hyperref,easybmat}
\usepackage{easybmat}
\usepackage{hyperref}
\usepackage[all]{xy}
\usepackage[active]{srcltx}

\usepackage{amsmath,amssymb,amsthm,mathtools}

\usepackage[active]{srcltx}
\usepackage{MnSymbol}
\usepackage{mathrsfs}
\DeclareMathOperator{\rank}{rank}

\newcommand{\sdotsss}%
{\text{\raisebox{-2.2pt}{$\cdot-$}%
 \raisebox{1.7pt}{$\cdot$}%
\raisebox{5.6pt}{$-\cdot$}}}

\renewcommand{\le}{\leqslant}
\renewcommand{\ge}{\geqslant}

\newtheorem{theorem}{Theorem}[section]

\newtheorem{lemma}[theorem]{Lemma}

\newtheorem{definition}[theorem]{Definition}

\newtheorem{remark}[theorem]{Remark}

\newtheorem{example}[theorem]{Example}

\usepackage{tcolorbox}
\tcbuselibrary{skins}
\tcbuselibrary{breakable}
\tcbset{shield externalize}

\newcommand{\hide}[1]{}

\begin{document}

\title{Canonical forms for pairs of matrices associated with Lagrangian and Dirac subspaces
}

\tnotetext[t1]{The work was supported by the Swedish Research Council (VR) under grant 2021-05393.}

\date{}

\author[um]{Sweta Das}
\ead{sweta.das@oru.se}

\author[vn,um]{Andrii Dmytryshyn}
\ead{andrii@chalmers.se}

\author[xo]{Volker Mehrmann}
\ead{mehrmann@math.tu-berlin.de}

\address[um]{School of Science and Technology, Örebro University, 70182 Örebro, Sweden}
\address[vn]{Department of Mathematical Sciences, Chalmers University of Technology and University of Gothenburg, 41296 Gothenburg, Sweden}
\address[xo]{Institut f{\"u}r Mathematik, Sekr. MA 4-5, TU Berlin, Straße des 17. Juni 136, 10623 Berlin, Germany}

\begin{abstract}
    We derive the canonical forms for a pair of $n\times n$ complex matrices $(E,Q)$ under transformations $(E,Q) \rightarrow (UEV,U^{-T}QV)$, and $(E,Q) \rightarrow (UEV,U^{-*}QV)$, where $U$ and $V$ are nonsingular complex matrices. We, in particular, consider the special cases of $E^TQ$ and $E^*Q$ being \\ (skew-)symmetric and (skew-)Hermitian, respectively,
    that are  associated with Lagrangian and Dirac subspaces and related linear-time invariant  dissipative Hamiltonian descriptor systems. 
\end{abstract}


\maketitle

\section{Introduction}
In this paper, we derive the canonical forms for a pair of $n\times n$ complex matrices $(E,Q)$ under the transformations $(E,Q) \rightarrow (UEV,U^{-T}QV)$ and $(E,Q) \rightarrow (UEV,U^{-*}QV)$, with two nonsingular complex matrices $U$ and $V$. 

Such pairs of matrices arise in the energy based modeling of linear time-invariant \textit{dissipative Hamiltonian descriptor systems} \cite{MeMW18} of the form
\begin{equation}\label{LTI_pH}
    E\dot{x} = (J-R)Qx,
\end{equation}
where $E$, $J$, $R$, $Q$ are $n\times n$ complex matrices $J = -J^{*}$ and $R = R^* \ge 0$, and $E^*Q=Q^*E\geq 0$, i.e., the columns of $\begin{bmatrix} E \\ Q \end{bmatrix}$ form a Lagrangian subspace. Note that here the notation $W\ge 0$ stands for $W$ being positive semidefinite.
See \cite{MehU23,MehS23,SchM23} for a detailed analysis of the relationship between Lagrange, respectively Dirac subspaces (the case that $EQ^*=-QE^*$) and port-Hamiltonian descriptor systems.

To understand the linear algebraic properties of  pH descriptor systems, one needs to conduct the analysis of the pencil $\lambda E - LQ$ related to \eqref{LTI_pH}, where $L = J-R$. For such an analysis it is beneficial to understand the properties of the pencil $\lambda E - Q$, which has been an active research are, see
\cite{MeMW18,MehS23}, where 
Kronecker-like condensed form of pencils $\lambda E- Q$ with some additional properties have been studied.

Canonical forms for pairs of matrices, possibly with additional structures, have been studied for various transformations, e.g., congruence $(A,B) \rightarrow (S^TAS,S^TBS)$ \cite{LaRo05, Thom91}, strict equivalence $(A,B) \rightarrow (SAL,SBL)$ \cite{Gant59,LaRo05, Rodm06, Thom91}, contragradient equivalence $(A,B) \rightarrow (SAL,L^{-1}BS^{-1})$ \cite{HoMe95}, simultaneous similarity $(A,B) \rightarrow (S^{-1}AS,S^{-1}BS)$ \cite{Hua23,Serg00}, $(A,H) \rightarrow (S^{-1}AS,S^{T}HS)$, for skew-symmetric $H$, $H$-symplectic $A$ \cite{GrJR}, feedback-injection equivalence $(A,B) \rightarrow (SAR,SBL)$, where $S$ and $L$ are block-triangular \cite{DmFR16,GaSe04}.

Most of the above canonical forms are sensitive to perturbations, see, e.g., \cite{Dmyt25}. One way to study the changes in the canonical structures that are arising from perturbations is by stratifying the orbits \cite{Dmyt17,DmKa14}. Codimensions of orbits play an important role in such stratifications
due to the fact that an orbit has only orbits of higher codimension in its closure \cite{Dmyt25,DmKS13}. Computation of codimensions of orbits has been carried out for generalized matrix products \cite{KaKK11}, contragredient matrix pencils \cite{GaMa99}, (*)congruence orbits of matrices \cite{DeDo11_1,DeDo11,DmFS12,DmFS14}, congruence orbits of (skew-)symmetric matrix pencils \cite{Dmyt19,Dmyt16,DmKS14}. 

The manuscript is organized as follows. Section $2$ recalls preliminary definitions, the canonical forms for square complex matrices under the transformations $^T$-congruence and $^*$-congruence along with the codimensions of orbits under the same group actions. In section $3$, we introduce an equivalence relation on pairs of complex matrices called $^T$-equivalence and $^*$-equivalence along with a necessary and sufficient condition for two pairs of matrices to be $^T$-equivalent and $^*$-equivalent. In Section $3.1$, we present the canonical forms under $^T$-equivalence and $^*$-equivalence for a complex pair $(E,Q)$ such that one of the matrices is nonsingular. Furthermore, we consider partial cases where $E^TQ$ and $E^*Q$ are symmetric (skew-symmetric) and Hermitian (skew-Hermitian), respectively. For  cases with both matrices possibly being singular, we derive canonical forms under the above-mentioned transformations, in Section $3.2$. Finally, we compute the codimensions of the $^T$-equivalence and $^*$-equivalence orbits in Sections $4.1$ and $4.2$, respectively.

\section{Preliminary results}
In this section we introduce the notation and present some preliminary results.

Define the matrices:
$$ \Gamma_n = 
    \begin{bmatrix}
        0 & & & & & (-1)^{n+1}\\
        & & & & \udots & (-1)^{n\phantom{+1}}\\
        & & & -1 & \udots & \\
        & & 1 & 1 & &  \\
        & -1 & -1 & & &  \\
        1 & 1 & & & & 0
    \end{bmatrix},\text{ } \Delta_n = 
    \begin{bmatrix}
         0 & & & 1\\
         & & \udots & i\\
         & 1 & \udots & \\
         1 & i & &  0 
    \end{bmatrix}, $$ 
    $$J_n(\lambda) =
    \begin{bmatrix}
        \lambda & 1 & & 0\\
        & \lambda & \ddots &\\
        & & \ddots & 1\\
        0 & & & \lambda
    \end{bmatrix}, \text{ and } H_{2n}(\mu) = 
    \begin{bmatrix}
    0 & I_n\\
    J_n(\mu) & 0
    \end{bmatrix}.
$$ The unspecified entries in the above matrices are zero. Note that $\Gamma_1 = [1], \text{ } \Delta_1 = [1], \text{ }J_1(\lambda) = [\lambda], \text{ }  H_{2}(\mu) =
    \begin{bmatrix}
    0 & 1\\
    \mu & 0
    \end{bmatrix}$. The sets $\mathbb{C}$, $\mathbb{R}$, $\mathbb{Z}$ and $\mathbb{Z}_+$ denote the complex, real, integer, and positive integer numbers, respectively. Also, $GL_{n}(\mathbb C)(\mathbb C)$ and $M_{m,n}(\mathbb C)$ are the sets of all nonsingular $n\times n$ complex matrices and $m\times n$ complex matrices, respectively. For convenience, we denote the set of complex square matrices by $M_{n}(\mathbb C)$.

Matrices $A$ and $B$ in $M_{n}(\mathbb C)$ are \emph{$^{T}$-congruent (respectively, $^{*}$-congruent)} if there exists $S \in GL_{n}(\mathbb C)$ such that $S^TAS =B$ (respectively, $S^*AS =B$). 

The following theorem (Theorem 1.1(a) in \cite{HoSe06}) presents a canonical form for square complex matrices under  $^{T}$-congruence.
\begin{theorem}
\label{congruence} Each square complex matrix is $^{T}$-congruent to a direct sum, uniquely determined up to permutation of summands, of canonical matrices of types:
    \begin{center}
    \begin{tabular}{|c|c|}
        \hline
        Type 0 & $J_n(0)$ \\
        \hline
        Type I & $\Gamma_{n}$ \\
        \hline
        Type II & $H_{2n}(\mu), 0 \neq \mu \neq (-1)^{n+1},$\\
        & $ \mu \text{ is determined up to replacement by } \mu^{-1}$\\
        \hline
    \end{tabular}
    \end{center}
\end{theorem}

The next theorem  (Theorem 1.1(b) in \cite{HoSe06}) presents a canonical form for square complex matrices under the transformation $^{*}$-congruence.
\begin{theorem}\label{*congruence}
    Each square complex matrix is $^{*}$-congruent to a direct sum, uniquely determined up to permutation of summands, of canonical matrices of types:
    \begin{center}
    \begin{tabular}{|c|c|}
        \hline
        Type 0 & $J_n(0)$ \\
        \hline
        Type I´ & $\lambda\Delta_{n}, \text{  } |\lambda| = 1$ \\
        \hline
        Type II´ & $H_{2n}(\mu), \text{  } |\mu| > 1$\\
        \hline
    \end{tabular}
    \end{center}
\end{theorem}

Consider in the following the action of $^{T}$-congruence of $GL_{n}(\mathbb C)$ on $M_{n}(\mathbb C)$, i.e., $S$ acts on $M$ by $S^TMS$. Then, $\mathcal{O}(M) = \{ S^TMS: S \in GL_{n}(\mathbb C)\}$ is the congruence orbit of $M$ under $^{T}$-congruence. The \emph{tangent space} of $\mathcal{O}(M)$ at the point $M$ is the set $\mathcal{T}(M) = \{ XM+MX^T: X \in M_{n}(\mathbb C)\}$. The codimension of the tangent space is equal to the dimension of its orthogonal complement which is also the codimension of the complex manifold $\mathcal{O}(M)$. It turns out to be the dimension of the solution space of the equation $XM+MX^T = 0$ \cite{DeDo11}. Similar definitions work for orbit under $^{*}$-congruence and its codimension, the only difference is that congruence orbit under the action of $^{*}$-congruence is a manifold over $\mathbb{R}$ instead of $\mathbb{C}$ \cite{DeDo11_1}. The following definitions play an important role in finding the codimensions of orbits of a matrix under the action of $^{T}$-congruence and $^{*}$-congruence.
\begin{definition}[Definition 2 in \cite{DeDo11}]\label{inter} 
    Let $M \in M_m, \text{ } N \in M_{n}(\mathbb C) $. The \textit{interaction} between $M$ and $N$, $\mbox{\rm \mbox{\rm inter}}(M,N)$, is the dimension of the solution space of the linear system
    \begin{equation*}
        XM + NY^T = 0, \text{ } YN + MX^T = 0,
    \end{equation*} for unknowns $X \in M_{n,m}, \text{ } Y \in M_{m,n}.$
\end{definition}
Considering the conjugate transpose instead of transpose in Definition~\ref{inter}, we have the following definition.
\begin{definition}[Definition 3.1 in \cite{DeDo11_1}]\label{inter*} 
   Let $M \in M_m, \text{ } N \in M_{n}(\mathbb C) $. The \emph{real \mbox{\rm inter}action} between $M$ and $N$, $\mbox{\rm \mbox{\rm inter}}^{*}(M,N)$, is the real dimension of the solution space over $\mathbb{R}$ of the linear system
    \begin{equation*}
        XM + NY^* = 0, \text{ } YN + MX^* = 0,
    \end{equation*} for unknowns $X \in M_{n,m}, \text{ } Y \in M_{m,n}.$
\end{definition}
Using the notion of \mbox{\rm inter}action, we have Theorem \ref{cod_cong} (Theorem 2 in \cite{DeDo11}), which presents the formula for the codimension of orbit of a square matrix under the action of $^{T}$-congruence.
\begin{theorem}\label{cod_cong} 
    The codimension of the orbit of a matrix $A \in M_{n}(\mathbb C)$ under the action of $^{T}$-congruence with canonical form $$ \mathcal{C}_A = J_{p_1}(0) \oplus ... \oplus J_{p_a}(0) \oplus \Gamma_{q_1} \oplus ... \oplus \Gamma_{q_b} \oplus H_{r_1}(\mu_1) \oplus ... \oplus H_{r_c}(\mu_c),$$ where $p_1 \ge p_2 \ge ... \ge p_a$ is computed as the following sum: $$ c_{A} = c_{0} + c_{1} + c_{2} + c_{00} + c_{11} + c_{22} + c_{01} + c_{02} + c_{12},$$ whose components are given by:
    \begin{enumerate}
         \item The codimension of the Type 0 blocks:
         $ c_0 = \sum_{i=1}^{a} \big\lceil\dfrac{p_i}{2}\big\rceil.$
         \item The codimension of the Type I blocks:
         $ c_1 = \sum_{i=1}^{b} \big\lfloor\dfrac{q_i}{2}\big\rfloor.$
         \item The codimension of the Type II blocks:
         $ c_2 = \sum_{i=1}^{c} r_i + 2\sum_j \big\lceil\dfrac{r_j}{2}\big\rceil,$ where the second sum is taken over those blocks $H_{2r_j}((-1)^{r_j})$ in $\mathcal{C}_A$.
         \item The codimension due to \mbox{\rm inter}actions between Type 0 blocks:
         $$c_{00} = \sum_{i,j=1, i<j}^{a} \mbox{\rm \mbox{\rm inter}}(J_{p_i}(0),J_{p_j}(0)),$$ where $\mbox{\rm \mbox{\rm inter}}(J_{p_i}(0),J_{p_j}(0)) =
         \begin{cases}
              p_j, \text{ if } p_j \text{ is even, }\\ 
              p_i, \text{ if } p_j \text{ is odd and } p_i \neq p_j,\\
              p_i + 1, \text{ if } p_j \text{ is odd and } p_i = p_j.
         \end{cases}$
         \item The codimension due to \mbox{\rm inter}actions between Type I blocks:\\
         $c_{11} = \sum \min\{q_{i},q_{j}\},$ where the sum runs over all pairs of blocks $(\Gamma_{q_i},\Gamma_{q_j})$, $ i<j,$ in $\mathcal{C}_A$ such that $q_i$ and $q_j$ have the parity (both odd or even).
         \item The codimension due to \mbox{\rm inter}action between Type II blocks:\\
         $c_{22} = 2\sum \min\{r_i,r_j\} + 4\sum \min\{r_s,r_t\},$ where the first sum is taken over all pairs $(H_{2r_i}(\mu_i), H_{2r_j}(\mu_j)), i<j$, of blocks in $\mathcal{C}_A$ such that $''\mu_i \neq \mu_j \text{ and } \mu_i\mu_j = 1''$ or $\mu_i = \mu_j \neq \pm 1$; and the second sum is taken over all pairs $(H_{2r_s}(\mu_s),H_{2r_t}(\mu_t))$, $s<t$, of blocks in $\mathcal{C}_A$ such that $\mu_s = \mu_t = \pm 1$.
         \item The codimension due to \mbox{\rm inter}actions between Type 0 and Type I blocks:
         $c_{01} = N_{odd}.\sum_{i=1}^{b}q_i,$ where $N_{odd}$ is the number of Type 0 blocks with odd size in $\mathcal{C}_A$.
         \item The codimension due to \mbox{\rm inter}actions between Type 0 and Type II blocks
         $c_{02} = N_{odd}.\sum_{i=1}^{c}r_i,$ where $N_{odd}$ is the number of Type 0 blocks with odd size in $\mathcal{C}_A)$.
         \item The codimension due to \mbox{\rm inter}actions between Type I and Type II blocks
         $c_{12} = 2\sum \min\{k,l\},$ where the sum is taken over all pairs $(\Gamma_k,H_{2l}((-1)^{k+1}))$ of blocks in $\mathcal{C}_A$.
     \end{enumerate}
\end{theorem}

The following theorem (Theorem 3.3 in \cite{DeDo11_1}) presents the (real) codimensions of orbits of a square matrix under the action of $^{*}$-congruence.
\begin{theorem}\label{*cong_cod}
     Let $A \in M_{n}(\mathbb C)$ with canonical form under the action of $^{*}$-congruence, $$ \mathcal{C}_A = J_{p_1}(0) \oplus ... \oplus J_{p_a}(0) \oplus \alpha_{q_1}\Delta_{q_1} \oplus ... \oplus \alpha_{q_b}\Delta_{q_b} \oplus H_{r_1}(\mu_1) \oplus ... \oplus H_{r_c}(\mu_c),$$ where $p_1 \ge p_2 \ge ... \ge p_a$. Then the real codimension of the orbit of A for the action of $^{*}$-congruence, depends only on $\mathcal{C}_A$. It can be computed  as the following sum: $$ c_{A} = c_{0} + c_{1} + c_{2} + c_{00} + c_{11} + c_{22} + c_{01} + c_{02},$$ whose components are given by: 
     \begin{enumerate}
         \item The codimension of the Type 0 blocks:
         $ c_0 = \sum_{i=1}^{a} 2\big\lceil\dfrac{p_i}{2}\big\rceil.$
         \item The codimension of the Type I´ blocks:
         $ c_1 = \sum_{i=1}^{b} q_i.$
         \item The codimension of the Type II´ blocks:
         $ c_2 = \sum_{i=1}^{c} 2r_i.$
         \item The codimension due to \mbox{\rm inter}actions between Type 0 blocks:
         $$c_{00} = \sum_{i,j=1, i<j}^{a} \mbox{\rm \mbox{\rm inter}}(J_{p_i}(0),J_{p_j}(0)),$$ where $\mbox{\rm \mbox{\rm inter}}(J_{p_i}(0),J_{p_j}(0)) =
         \begin{cases}
              2p_j, \text{ if } p_j \text{ is even, }\\ 
              2p_i, \text{ if } p_j \text{ is odd and } p_i \neq p_j,\\
              2(p_i + 1), \text{ if } p_j \text{ is odd and } p_i = p_j.
         \end{cases}$
         \item The codimension due to \mbox{\rm inter}actions between Type I´ blocks:\\
         $c_{11} = \sum \min\{q_{i},q_{j}\},$ where the sum runs over all pairs of blocks $(\Delta_{q_i},\Delta_{q_j})$, $ i<j,$ in $\mathcal{C}_A$ such that (a) $q_i$ and $q_j$ have the same parity (both odd or both even) and $\alpha_i= \pm \alpha_j$, and (b) $q_i$ and $q_j$ have different parity and $\alpha_i= \pm \alpha_j$.
         \item The codimension due to \mbox{\rm inter}action between Type II´ blocks:\\
         $c_{22} = 4\sum \min\{r_i,r_j\},$ where the first sum is taken over all pairs  $(H_{2r_i}(\mu_i),$ $ H_{2r_j}(\mu_j)), i<j$, of blocks in $\mathcal{C}_A$ such that such that $\mu_i=\mu_j$.
         \item The codimension due to \mbox{\rm inter}actions between Type 0 and Type I´ blocks:
         $c_{01} = N_{odd}.\sum_{i=1}^{b}2q_i,$ where $N_{odd}$ is the number of Type 0 blocks with odd size in $\mathcal{C}_A$.
         \item The codimension due to \mbox{\rm inter}actions between Type 0 and Type II´ blocks
         $c_{02} = N_{odd}.\sum_{i=1}^{c}4r_i,$ where $N_{odd}$ is the number of Type 0 blocks with odd size in $\mathcal{C}_A$.
     \end{enumerate}
\end{theorem}

After presenting the notation and some preliminary results, in the next section we derive the canonical forms for our special structured pairs.

\section{$^{T}$-equivalence and $^{*}$-equivalence for a pair of matrices}

In this section we consider pairs of matrices $(E,Q)$ and the canonical forms under the following equivalence relations.
A pair of matrices $(E,Q) \in (M_{m \times n}(\mathbb C),M_{m \times n}(\mathbb C))$ is said to be \emph{$^{T}$-equivalent} (respectively, \emph{$^{*}$-equivalent}) to ($E_1$,$Q_1$) if there exist $U \in GL_{m}(\mathbb C)$, $V  \in GL_{n}(\mathbb C)$ such that $(E_1,Q_1) = (UEV,U^{-T}QV)$ (respectively, $(E_1,Q_1) = (UEV,U^{-*}QV)$).
%
%
%
We now consider two cases. First we study the case where at least one of $E$ and $Q$ is nonsingular and after this the case that both $E$ and $Q$ are singular.

\subsection{Canonical forms for $E$ or $Q$ nonsingular}
The case $Q = I$ in \eqref{LTI_pH} arises in a variety of applications, e.g., modeling of a simple RLC network \cite[Example 1]{MeMW18}, space discretization of the Stokes and Oseen equations in fluid dynamics \cite[Example 2]{MeMW18}, space discretization of the Euler equation describing the flow in a gas network \cite[Example 3]{MeMW18}. So, in this section, we present a canonical form for $(E,Q)\in (M_{n}(\mathbb C),M_{n}(\mathbb C))$  under $^{T}$-equivalence and $^{*}$-equivalence transformations, in the case when at least one of $E$ and $Q$ is nonsingular which generalizes the situation that $Q=I$. Before that, we present  necessary and sufficient conditions for two pairs of complex matrices to be $^{T}$-equivalent and $^{*}$-equivalent.

\begin{lemma}\label{condition}
Let $E,Q, E_1,Q_1\in M_{n}(\mathbb C)$ and $E$ or $Q$ be a nonsingular matrix. Then, the following statements are equivalent.
\begin{enumerate}
    \item[(a)] $(E,Q)$ is $^{T}$-equivalent (respectively, $^{*}$-equivalent) to ($E_1$,$Q_1$).
    \item[(b)] $E^{T}Q$ (respectively, $E^{*}Q$) is $^{T}$-congruent (respectively, $^{*}$-congruent) to $E_{1}^{T}Q_{1}$ (respectively, $E_{1}^{*}Q_{1}$).
    \item[(c)] $EQ^{-1}$ is $^{T}$-congruent (respectively, $^{*}$-congruent) to $E_{1}Q_{1}^{-1}$, provided that $Q$ is
    nonsingular.
\end{enumerate}
\end{lemma}
\begin{proof} We first prove that {\it (a)} is equivalent to {\it (b)} and then prove the equivalence of {\it (a)} and {\it (c)}. It is clear that the nonsingularity of $E_1$ or $Q_1$ follow from the nonsingularity of $E$ or $Q$, respectively. To show that {\it (a)} implies {\it (b)}, assume $(E,Q)$ is $^T$-equivalent to $(E_1,Q_1)$, so there exist $U,V \in GL_{n}(\mathbb C)$ such that $(E_1,Q_1) = (UEV,U^{-T}QV)$. Then $E_{1}^{T}Q_{1} = V^{T}(E^{T}Q)V$ and hence, $E^{T}Q$ is $^{T}$-congruent to $E_{1}^{T}Q_{1}$.

To prove the converse, we first consider a nonsingular $Q$. Since $E^TQ$ is $^T$-congruent to $E_1^TQ_1$, there exists $Y \in GL_{n}(\mathbb C)$
such that $E^{T}Q = YE_{1}^{T}Q_{1}Y^{T}$, or equivalently, $E^{T} = YE_{1}^{T}Q_{1}Y^{T}Q^{-1}$, i.e., $E = (Q^{-T}YQ_{1}^{T})E_{1}Y^{T}.$ Note that $Q$ can be written as $QY^{-T}Y^{T} = QY^{-T}Q_{1}^{-1}Q_{1}Y^{T} = (Q^{-T}YQ_{1}^{T})^{-T}Q_{1}Y^{T}.$ Thus $(E,Q)$ is $^{T}$-equivalent to ($E_1$,$Q_1$) via $Q^{-T}YQ_{1}^{T}$ and $Y^{T}$.

For the case when $E$ is nonsingular, $^T$-congruency of $E^TQ$ and $E_1^TQ_1$ states the existence of $Z \in GL_{n}(\mathbb C)$ such that $E^{T}Q = ZE_{1}^{T}Q_{1}Z^{T}$, or equivalently, $Q = (E^{-T}ZE_{1}^{T})Q_{1}Z^{T} = (EZ^{-T}E_{1}^{-1})^{-T}Q_{1}Z^{T}$. Note that $E$ can be written as $(EZ^{-T}E_{1}^{-1})E_{1}Z^{T}$. Thus, $(E,Q)$ is $^{T}$-equivalent to ($E_1$,$Q_1$) via $EZ^{-T}E_{1}^{-1}$ and $Z^{T}$.

It is straightforward to see that  $(a)$ implies $(c)$. For the reverse implication, assume that $EQ^{-1}$ is $^T$-congruent to $E_1Q^{-1}_1$. Then there exists $Z \in GL_{n}(\mathbb C)$ such that $EQ^{-1} = ZE_{1}Q_{1}^{-1}Z^{T}$, or equivalently, $E = ZE_{1}(Q_{1}^{-1}Z^{T}Q)$. Moreover, $Q$ can be written as $Q = Z^{-T}Q_{1}(Q_{1}^{-1}Z^{T}Q)$, and thus, $(E,Q)$ is $^{T}$-equivalent to ($E_1$,$Q_1$) via $Z$ and $Q_{1}^{-1}Z^{T}Q$.

The proof for the necessary and sufficient condition for $^{*}$-equivalence is analogous  by considering conjugate transpose instead of transpose of matrix and $^{*}$-congruence instead of $^{T}$-congruence in the above arguments. 
\end{proof}

The following Theorems~\ref{CCF_*_equi} and \ref{CCF_*_equi_2}  provide canonical forms for a pair of matrices under $^{T}$-equivalence and $^{*}$-equivalence, respectively, in the case when one of the matrices is nonsingular.

\begin{theorem}[\textbf{Canonical form under $^{T}$-equivalence}]\label{CCF_*_equi}
Let $(E,Q)\in (M_{n}(\mathbb C)$, $M_{n}(\mathbb C))$ and assume that at least one of $E$ and $Q$ is nonsingular. Then there exist $U,V\in GL_{n}(\mathbb C)$ such that for nonsingular $E$,
    $$(UEV,U^{-T}QV) = \big(I_n,(\bigoplus_{i = 1}^{p} H_{2m_i}(\mu_i)) \oplus (\bigoplus_{j = 1}^{q} \Gamma_{n_j})\oplus (\bigoplus_{k = 1}^{r}J_{r_k}(0))\big), \text{ and}$$
for nonsingular $Q$,
\begin{equation}\label{to_prove}
    (UEV,U^{-T}QV) = \big(I_{n-v} \oplus (\bigoplus_{k = 1}^{r}J_{r_k}(0)^{T}),(\bigoplus_{i = 1}^{p} H_{2m_i}(\mu_i)) \oplus (\bigoplus_{j = 1}^{q} \Gamma_{n_j}) \oplus I_v\big),
\end{equation}
    where $v = \sum_{k=1}^{r}r_k$, $ 0 \neq \mu \neq (-1)^{n+1}$, and $\mu $ is determined up to replacement by $\mu^{-1}$. The canonical form is uniquely determined up to permutations of pairs of summands.
\end{theorem}

\begin{proof} 
Let $E,Q\in M_{n}(\mathbb C)$ and assume that $Q$ is nonsingular. We prove the existence of $U,V\in GL_{n}(\mathbb C)$ such that \eqref{to_prove} holds. The same arguments will hold for the case that $E$ is nonsingular.

Let the congruence canonical form of $E^TQ$ be given by
$$\mathcal{C}_{E^TQ} = (\bigoplus_{i = 1}^{p} H_{2m_i}(\mu_i)) \oplus (\bigoplus_{j = 1}^{q} \Gamma_{n_j}) \oplus (\bigoplus_{k = 1}^{p}J_{r_k}(0)).$$ Define $M = I_{n-v} \oplus (\bigoplus_{k = 1}^{r}J_{r_k}(0)^{T}),$ and $N = (\bigoplus_{i = 1}^{p} H_{2m_i}(\mu_i)) \oplus (\bigoplus_{j = 1}^{q} \Gamma_{n_j}) \oplus I_v, v = \sum_{k=1}^{r}r_k, 0 \neq \mu \neq (-1)^{n+1}$, where $\mu  $ is determined up to replacement by $ \mu^{-1}$.
        
Note that $\mathcal{C}_{E^TQ} = M^TN$, and by Theorem \ref{congruence}, $M^TN$ is $^{T}$-congruent to $E^TQ$, so by Lemma \ref{condition}, $(E,Q)$ is $^{T}$-equivalent to $(M,N)$.

{\bf{Uniqueness}}: We prove uniqueness for the case of nonsingular $Q$. The same arguments hold if $E$ is nonsingular. Suppose that there exist $U,V\in GL_{n}(\mathbb C)$ such that $$(UEV,U^{-T}QV) = (I_{n-v} \oplus (\bigoplus_{k = 1}^{r}J_{r_k}(0)^{T}), (\bigoplus_{i = 1}^{p} H_{2m_i}(\mu_i)) \oplus (\bigoplus_{j = 1}^{q} \Gamma_{n_j}) \oplus (\bigoplus_{k = 1}^{r}I_{r_k})),$$
where $v = \sum_{k=1}^{r}r_k, 0 \neq \mu \neq (-1)^{n+1},\text{ and } \mu \text{ is determined up to replacement by } \mu^{-1}.$ If $(E,Q)$ has another canonical form under the same transformation given by
$$
(I_{n-v'} \oplus (\bigoplus_{k = 1}^{r'}J_{r_{k}}(0)^{T}),(\bigoplus_{i = 1}^{p'} H_{2m_{i}}(\mu_{i})) \oplus (\bigoplus_{j = 1}^{q'} \Gamma_{n_{j}}) \oplus (\bigoplus_{k = 1}^{r'}I_{r_{k}}),
$$
then $E^TQ$ is $^{T}$-congruent to $(\bigoplus_{{i} = 1}^{p'} (H_{2m_{i}}(\mu_{i})) \oplus (\bigoplus_{{j} = 1}^{q'} \Gamma_{n_{j}}) \oplus (\bigoplus_{k = 1}^{r'}J_{r_{k}}(0))$. But since the canonical form of $E^TQ$ under $^{T}$-congruence is unique, up to permutation of its canonical blocks, it follows that $v=v'$, $p' = p$, $q'=q$, \text{ and } $r' = r$. Thus, the canonical form is unique up to permutation of canonical pairs.
\end{proof}
The following remark provides another option for choosing some canonical summands in Theorem \ref{CCF_*_equi}. This option avoids using $\Gamma_{k}$ and $H_{2k}(\mu_{i})$, and uses only the Jordan blocks and the skew-identity matrices (up to signs).
\begin{remark}\label{alt_can}
The pairs $(I_k,\Gamma_k)$ and $(I_{2k},H_{2k}(\mu))$ in Theorem~\ref{CCF_*_equi} are $^T$-equivalent to
$$
\left(\begin{bmatrix}
        & & 1\\
        & \udots & \\
        (-1)^{k+1} & &
    \end{bmatrix}, J_k(1) \right) \ \text{and} \ 
    \left(\begin{bmatrix}
        & & 1\\
        & \udots & \\
        1 & &
    \end{bmatrix},J_{k}(\mu) \oplus I_k \right),
$$
respectively, 
\end{remark}
Similarly to Theorem \ref{CCF_*_equi} but considering the conjugate transpose instead of the transpose and $^*$-congruence instead of $^T$-congruence, the following result is immediate.
\begin{theorem}[\textbf{Canonical form under $^{*}$-equivalence}]\label{CCF_*_equi_2}
Let $(E,Q)\in (M_{n}(\mathbb C),M_{n}(\mathbb C))$ and at least one of $E$, $Q$ is nonsingular. Then there exist $U,V\in GL_{n}(\mathbb C)$ such that for nonsingular $E$,
$$
(UEV,U^{-*}QV) = (I_n,(\bigoplus_{i = 1}^{p} H_{2m_i}(\mu_i)) \oplus (\bigoplus_{j = 1}^{q}e^{i\theta_{j}} \Delta_{n_j})\oplus (\bigoplus_{k = 1}^{s}J_{r_k}(0)),
$$
and for nonsingular $Q$,
$$
(UEV,U^{-*}QV) = (I_{n-v} \oplus (\bigoplus_{k = 1}^{s}J_{r_k}(0)^{*}),(\bigoplus_{i = 1}^{p} H_{2m_i}(\mu_i)) \oplus (\bigoplus_{j = 1}^{q}e^{i\theta_{j}} \Delta_{n_j})\oplus I_v),
$$ 
where $v = \sum_{k=1}^{s}r_k$, $|\mu_i| > 1,\text{ } 0\le \theta_j < 2\pi$. The canonical form is uniquely determined up to permutations of pairs of summands.
\end{theorem}
In this section we have derived canonical forms for pairs $(E,Q)$ where either $E$ or $Q$ is nonsingular. In the following subsection we specialize these results for the case of structured products. 
\subsubsection{Structured $E^{*}Q$/$E^{T}Q$}
As discussed earlier, the Hamiltonian $\mathcal{H}(x) = \dfrac{1}{2}x^{*}E^{*}Qx$ describes the distribution of \mbox{\rm inter}nal energy among energy storage elements of the system in \eqref{LTI_pH}, so, $E^{*}Q$ is typically assumed to be Hermitian for $\mathcal{H}(x)$ to be real. On the other hand in the context of port-Hamiltonian modelling via Dirac structures, the case that $EQ^*$ is 
skew-Hermitian arises, see \cite{SchM18,SchM23}.

In the following, we therefore present Theorems \ref{E^TQ is symmetric}-\ref{E^*Q_is_sk-hermi}, that are special cases of Theorems \ref{CCF_*_equi} and \ref{CCF_*_equi_2}, for $E^{T}Q/E^{*}Q$ being (skew-)symmetric res\-pectively (skew-)Hermitian. In Section \ref{sec:sing} we extend these results to the case of singular $E$ and $Q$.  We begin with the
$^T$ case.
%
\begin{theorem}\label{E^TQ is symmetric}
Let $(E,Q)\in (M_{n}(\mathbb C),M_{n}(\mathbb C))$ with $E^TQ$  symmetric and at least one of $E$ and $Q$ is nonsingular. Then there exist $U,V\in GL_{n}(\mathbb C)$, such that
\begin{multline*}
(UEV,U^{-T}QV) = \bigoplus_{i=1}^n(A_i,B_i), \\
\text{where }  (A_i,B_i ) = 
\begin{cases}
    (1,1),(1,0), & \text{for nonsingular $E$, }\\
    (1,1), (0,1), & \text{ for nonsingular $Q$.}
\end{cases} 
\end{multline*}
These canonical forms are uniquely determined up to permutation of summands.
\end{theorem}
\begin{proof}
    We prove the result for the case of nonsingular $Q$. The result for the case of nonsingular $E$ follows using similar arguments. $E^TQ$ is $^{T}$-congruent to its canonical congruence form, $\mathcal{C}_{E^TQ} = (\bigoplus_{i = 1}^{p} H_{2m_i}(\mu_i)) \oplus (\bigoplus_{j = 1}^{q} \Gamma_{n_j}) \oplus (\bigoplus_{k = 1}^{r}J_{r_k}(0)), 0 \neq \mu \neq (-1)^{m_i+1}$, where $ \mu$  is determined up to replacement by $\mu^{-1}.$ Let $v = \sum_{k = 1}^r r_k$, i.e., $v = n-\rank(E^TQ)$. Since $\mathcal{C}_{E^TQ}$ is symmetric, $\mathcal{C}_{E^TQ} = (\bigoplus_{j = 1}^{q}\Gamma_1) \oplus (\bigoplus_{k = 1}^{v}J_{1}(0))$ which can be written as $I_{n-v} \oplus \begin{bmatrix}
        0
    \end{bmatrix}_v$. Setting $N = I_{n},$ $M = I_{n-v} \oplus (\bigoplus_{k = 1}^{v}J_{1}(0))$, we get $\mathcal{C}_{E^TQ} = M^TN$ and hence the result follows as in Theorem \ref{CCF_*_equi}.
\end{proof}
Next we derive a result analogous to Theorem \ref{E^TQ is symmetric} but under $^*$-equivalence instead of $^T$-equivalence. Furthermore, the special case of $E^*Q$ being positive semi-definite is also considered because of its importance in port-Hamiltonian systems with nonnegative Hamiltonian and in classical first order representations of linear damped mechanical systems \cite{MeMW18,TiMe01,Vese11}. The canonical pairs are presented in Theorems \ref{E^*Q is hermitian} and \ref{E^*Q_is_sk-hermi}. 
\begin{theorem}\label{E^*Q is hermitian}
    Let $(E, Q) \in (M_{n}(\mathbb C),M_{n}(\mathbb C))$ with $E^*Q$  Hermitian and at least one of $E$, $Q$ is nonsingular.  There exist $U,V \in GL_{n}(\mathbb C)$ such that
\begin{multline*}
    (UEV,U^{-*}QV) = \bigoplus_{i=1}^n(A_i,B_i),\\
\text{where }  (A_i,B_i ) = 
\begin{cases}
    (1,1),(1,-1),(1,0), & \text{for nonsingular $E$,}\\
    (1,1),(1,-1), (0,1), & \text{ for nonsingular $Q$.}
\end{cases}
\end{multline*}
In addition, if $E^*Q$ is positive semi-definite, then $(1,-1)$ is no longer a canonical block. The above canonical form is uniquely determined up to permutation of summands.
\end{theorem}
\begin{proof}
The proof is similar to the proof method of Theorem \ref{E^TQ is symmetric}, but here we consider the conjugate transpose of a matrix instead of its transpose and $^*$-congruence instead of $^T$-congruence. Thus, we restrict ourselves to Hermitian canonical blocks in Theorem \ref{CCF_*_equi_2}. Note, that $\Delta_n$ is symmetric but not Hermitian. Also, $e^{i\theta_{j}} = \overline{e^{i\theta_{j}}}$ if and only if the imaginary part of the complex number is zero, i.e., $\theta_{j} = n\pi$, $n$ is an integer. But by Theorem \ref{CCF_*_equi_2}, $\theta_{j} \in [0,2\pi)$ which implies that $\theta_{j} \in \{0, \pi\}$.  Thus, after permuting the rows and respective columns (if needed), we have, $\bigoplus_{j = 1}^{n-v}e^{i\theta_{j}} = I_x \oplus (-I_y)$, where $x+y = n-v$ and $v =\rank(E^*Q)$.

For the special case that $E^*Q$ is positive semi-definite and $Q$ is nonsingular, from the previous arguments, there exist $U,V \in GL_{n}(\mathbb C)$ such that $(UEV,U^{-*}QV) = (I_{n-v} \oplus [0]_v,(\bigoplus_{j = 1}^{n-v}e^{i\theta_{j}}) \oplus I_{v})$,  where $\theta_j \in \{0,\pi\}$. Thus, $E^*Q = V^{-*}((\bigoplus_{j = 1}^{n-v}e^{i\theta_{j}}) \oplus [0]_v)V^{-1}$, i.e., for all $x\in \mathbb{C}^n$, $x^*E^*Qx = x^*V^{-*}((\bigoplus_{j = 1}^{n-v}e^{i\theta_{j}}) \oplus [0]_v)V^{-1}x = y^{*}((\bigoplus_{j = 1}^{n-v}e^{i\theta_{j}}) \oplus [0]_v)y 
    =  e^{i\theta_{1}}|y_1|^2 + \dots + e^{i\theta_{n-v}}|y_{n-v}|^2$, where $y = V^{-1}x$ and $y = [y_i]_{i = 1}^n$. For all $x \in \mathbb{C}^n$, the above equality is nonnegative if $e^{i\theta_{j}} = 1$, i.e., $\theta_{j} = 0$, for all $j$. Similar arguments hold when $E$ is nonsingular. Thus, the assertion follows.
\end{proof}

Using similar arguments as in the previous two theorems, we provide a canonical form when $E^TQ$ is skew-symmetric. 
\begin{theorem}\label{K,L_can_form}
    Let $(E,Q)\in (M_{n}(\mathbb C),M_{n}(\mathbb C))$, $E^TQ$  skew-symmetric and at least one of $E$, $Q$ is nonsingular. Then there exist $U,V\in GL_{n}$ such that
    \begin{multline*}(UEV,U^{-T}QV) = \bigoplus_{i=1}^f(A_i,B_i), \\
    \text{where } (A_i,B_i) \in \begin{cases}
        (I_2,H_2(-1)),(1,0), & \text{  for nonsingular $E$,}\\
        (I_2,H_2(-1)),(0,1), & \text{  for nonsingular $Q$,}
    \end{cases}  \text{ and } f\in \mathbb{Z}_{+}.
    \end{multline*}
    This canonical form is uniquely determined up to permutation of summands.
\end{theorem}
\begin{proof}
    We prove the result for the case of nonsingular $Q$. The result for the other case follows similarly.  Following the proof method of Theorem \ref{E^TQ is symmetric} and using the skew-symmetry, we get that $E^TQ$ is $^{T}$-congruent to $\mathcal{C}_{E^TQ} = (\bigoplus_{j = 1}^{p} H_{2}(\mu)) \oplus (\bigoplus_{k = 1}^{v}J_{1}(0))$, $\mu = -1$, $2p =\rank(E^TQ)$ and $v = n-2p$. Defining $N = (\bigoplus_{j = 1}^{p} H_{2}(-1)) \oplus I_{v},$ $M = I_{2p} \oplus [0]_v$, we get $\mathcal{C}_{E^TQ} = M^TN$ and hence the result follows as in Theorem \ref{CCF_*_equi}.
\end{proof}

Result analogous to Theorem \ref{K,L_can_form} under $^*$-equivalence instead of $^T$-equivalence is presented in the next theorem.
\begin{theorem}\label{E^*Q_is_sk-hermi}
    Let $(E, Q) \in (M_{n}(\mathbb C),M_{n}(\mathbb C))$ and $E^*Q$ is skew-Hermitian and at least one of $E$ or $Q$ is nonsingular.  There exist $U,V \in GL_{n}(\mathbb C)$ such that
    \begin{multline*}
        (UEV,U^{-*}QV) = \bigoplus_{i=1}^n(A_i,B_i),\\
        \text{where } (A_i,B_i) \in \begin{cases}
        (1,i), (1,-i), (1,0), & \text{ for nonsingular $E$,}\\
        (1,i), (1,-i), (0,1), & \text{ for nonsingular $Q$.}
    \end{cases} 
    \end{multline*}
     The above canonical form is uniquely determined up to permutation of summands. 
\end{theorem}
\begin{proof}
    Using similar arguments as in the proof of Theorem \ref{K,L_can_form}, we restrict ourselves to skew-Hermitian canonical blocks in Theorem \ref{CCF_*_equi_2}. Note that $\Delta_{n_j}$ in Theorem \ref{CCF_*_equi_2} is not skew-Hermitian and $e^{i\theta_{j}} = -\overline{e^{i\theta_{j}}}$ if and only if $\theta_{j} = \frac{(n+1)\pi}{2}$, where $n \in 2\mathbb{Z}$. But by Theorem \ref{CCF_*_equi_2}, $\theta_{j} \in [0,2\pi)$ implying that $\theta_{j} \in \{\frac{\pi}{2},\frac{3\pi}{2}\}$.
\end{proof}
In this subsection we have presented the canonical forms for the situation that one the matrices $E$ or $Q$ are nonsingular. In the next section we extend the results to the case that both $E$ and $Q$ are allowed to be singular.
\subsection{Canonical forms for both $E$ and $Q$  singular}
\label{sec:sing}
Pencils where both $E$ and $Q$ may are singular related to \eqref{LTI_pH} typically arise in applications \cite{GMQSW96,MehS23}, as limiting cases or due to redundancy in system modeling. For this situation we now develop results that allow us to formulate analogues of Theorems \ref{CCF_*_equi} and \ref{CCF_*_equi_2} for singular $E$ and $Q$, and structured $E^TQ/E^*Q$. We start by recalling an auxiliary lemma.
\begin{lemma}\cite{HoMe95}\label{lemma_nonsing}
    Let positive integers $k$, $n$ be given with $k<n$. Let $Y_1 \in M_{n,k}$ and $P\in M_{k,n}$ be given with $PY_1$ nonsingular. Then there exists a $Y_2 \in M_{n,n-k}$ such that $PY_2 = 0 $ and $\begin{bmatrix}
        Y_1 & Y_2
    \end{bmatrix} \in M_{n}(\mathbb C)$ is nonsingular.
\end{lemma}

Then we have the following Lemma~\ref{lemma_true_for_all}.
\begin{lemma}\label{lemma_true_for_all}
    Let $(E,Q) \in (M_{n}(\mathbb C),M_{n}(\mathbb C))$, $0 <\rank(E^TQ) = k < n$. Then there exist $U,V \in GL_{n}(\mathbb C)$ such that 
    $$(UEV,U^{-T}QV) = (I_k \oplus \mathcal{E}, (\bigoplus_{i = 1}^{p} H_{2m_i}(\mu_i)) \oplus (\bigoplus_{j = 1}^{q} \Gamma_{n_j}) \oplus \mathcal{Q}),$$
    where $\mathcal{E}^T\mathcal{Q} = \bigoplus_{s = 1}^{p}J_{r_s}(0)$, $ 0 \neq \mu \neq (-1)^{n+1}$ and $\mu $ is determined up to replacement by $\mu^{-1}$. 
\end{lemma}
\begin{proof}
    For the sake of simplicity of terms, assume the canonical congruence form of $E^TQ$ is written as $\mathcal{C}_{E^TQ} = \mathcal{R}_{E^TQ} \oplus \mathcal{N}_{E^TQ}$, where $\mathcal{R}_{E^TQ}$ is the regular part and $\mathcal{N}_{E^TQ}$ corresponds to the nilpotent part. Let $V \in GL_{n}(\mathbb C)$ be such that $V^TE^TQV = \mathcal{R}_{E^{T}Q} \oplus \mathcal{N}_{E^{T}Q}$. Partition $V^TE^T$ and $QV$, respectively, as
    \begin{equation}\label{V^TE^T}
        V^TE^T = 
    \begin{bmatrix}
        C_1^T\\
        C_2^T
    \end{bmatrix} \text{ with } C_1 \in M_{n,k}, C_2 \in M_{n,n-k}
    \end{equation} 
    \begin{equation}\label{QV}
    \text{and } QV = 
    \begin{bmatrix}
        D_1 & D_2
    \end{bmatrix} \text{ with } D_1 \in M_{n,k}, D_2 \in M_{n,n-k}.
    \end{equation}
    From \eqref{V^TE^T} and \eqref{QV} we have
    $$\mathcal{R}_{E^{T}Q} \oplus \mathcal{N}_{E^{T}Q} = V^TE^TQV = \begin{bmatrix}
        C_1^T\\
        C_2^T
    \end{bmatrix}\begin{bmatrix}
        D_1 & D_2
    \end{bmatrix} = \begin{bmatrix}
        C_1^TD_1 & C_1^TD_2 \\
        C_2^TD_1 & C_2^TD_2
    \end{bmatrix}.$$
    Equating both sides in this equation, we get $C_1^TD_2 = 0$, $C_2^TD_1 = 0$, $C_2^TD_2 = \mathcal{N}_{E^{T}Q}$ and $C_1^TD_1 = \mathcal{R}_{E^{T}Q}$ is nonsingular and so is $D_1^TC_1$. By Lemma \ref{lemma_nonsing} there exist $C_3$ and $D_3$ such that 
    $F = \begin{bmatrix}
        C_1 & C_3
    \end{bmatrix} \in GL_{n}(\mathbb C)$, 
    $U = \begin{bmatrix}
        D_1 & D_3
    \end{bmatrix} \in GL_{n}(\mathbb C)$ and $C_1^TD_3 = 0$, $D_1^TC_3 = 0$, i.e., $C_3^TD_1 = 0.$ Note that
    \begin{equation}\label{F^TU}
    F^TU = 
    \begin{bmatrix}
        C_1^TD_1 & C_1^TD_3 \\
        C_3^TD_1 & C_3^TD_3
    \end{bmatrix} = 
    \begin{bmatrix}
        \mathcal{R}_{E^{T}Q} & 0 \\
        0 & C_3^TD_3
    \end{bmatrix}
    \end{equation}
    Since $F^TU$ is nonsingular, so is $C_3^TD_3$. Using \eqref{V^TE^T} and the expression for $U$, we have,  \begin{equation}\label{V^TE^TU}
    V^TE^TU = 
    \begin{bmatrix}
        C_1^TD_1 & C_1^TD_3 \\
        C_2^TD_1 & C_2^TD_3
    \end{bmatrix} =
    \begin{bmatrix}
        \mathcal{R}_{E^{T}Q} & 0 \\
        0 & C_2^TD_3
    \end{bmatrix}.
    \end{equation}
    Then  $V^TE^TF^{-T} = V^TE^TUU^{-1}F^{-T} = (V^TE^TU)(F^TU)^{-1}$, and using \eqref{F^TU} and \eqref{V^TE^TU}, we have
    \begin{equation}\label{V^TE^TF^-T}
    V^TE^TF^{-T} = 
    \begin{bmatrix}
        \mathcal{R}_{E^{T}Q} & 0 \\
        0 & C_2^TD_3
    \end{bmatrix}
    \begin{bmatrix}
        \mathcal{R}_{E^{T}Q}^{-1} & 0 \\
        0 & (C_3^TD_3)^{-1}
    \end{bmatrix} = 
    \begin{bmatrix}
        I_k & 0\\
        0 & (C_2^TD_3)(C_3^TD_3)^{-1}
    \end{bmatrix}.\end{equation}
    Transposing both sides of \eqref{V^TE^TF^-T} we get, \begin{equation}\label{F^{-1}EY}
        F^{-1}EV = 
    \begin{bmatrix}
        I_k & 0\\
        0 & ((C_2^TD_3)(C_3^TD_3)^{-1})^T
    \end{bmatrix} = :
    \begin{bmatrix}
        I_k & 0\\
        0 & \mathcal{E}
    \end{bmatrix}.
    \end{equation}
    Again, using \eqref{QV} and the expression for $F$, we obtain
    \begin{equation}\label{F^TQV}
        F^TQV = 
    \begin{bmatrix}
        C_1^TD_1 & C_1^TD_2 \\
        C_3^TD_1 & C_3^TD_2
    \end{bmatrix} =
    \begin{bmatrix}
        \mathcal{R}_{E^{T}Q} & 0 \\
        0 & C_3^TD_2
    \end{bmatrix} = : \begin{bmatrix}
        \mathcal{R}_{E^{T}Q} & 0 \\
        0 & \mathcal{Q}
    \end{bmatrix}.
    \end{equation} 
    Then, the assertion follows from the fact that  $\mathcal{N}_{E^{T}Q} = \mathcal{E}^T\mathcal{Q}$ is nilpotent, and
    \begin{eqnarray*}\mathcal{R}_{E^{T}Q} \oplus \mathcal{N}_{E^{T}Q} &=& V^TE^TQV = (V^TE^TF^{-T})(F^TQV) = \begin{bmatrix}
        I_k & 0\\
        0 & \mathcal{E}^T
    \end{bmatrix}
    \begin{bmatrix}
        \mathcal{R}_{E^{T}Q} & 0 \\
        0 & \mathcal{Q}
    \end{bmatrix}\\
    &= &
    \begin{bmatrix}
        \mathcal{R}_{E^{T}Q} & 0 \\
        0 & \mathcal{E}^T\mathcal{Q}
    \end{bmatrix}.
    \end{eqnarray*}
\end{proof}
Using Theorem \ref{CCF_*_equi_2}, similar arguments as in Lemma \ref{lemma_true_for_all} and considering the conjugate transpose  of matrices instead of the transpose in Lemma \ref{lemma_true_for_all}, we get the following analogous result.
\begin{lemma}\label{lemma_true_all_*}
    Let $(E,Q) \in (M_{n}(\mathbb C),M_{n}(\mathbb C))$. Then there exist $U,V \in GL_{n}(\mathbb C)$ such that 
    $$(UEV,U^{-*}QV) = (I_{n-v} \oplus \mathcal{E}, (\bigoplus_{i = 1}^{p} H_{2m_i}(\mu_i)) \oplus (\bigoplus_{j = 1}^{q}e^{i\theta_{j}} \Delta_{n_j}) \oplus \mathcal{Q}),$$
    where $v = \sum_{k=1}^{s}r_k$, $|\mu_i| > 1,\text{ } 0\le \theta_j < 2\pi$ and $\mathcal{E}^*\mathcal{Q} = \bigoplus_{s = 1}^{p}J_{r_s}(0)$.
\end{lemma}
The following remark suggests an approach that leads us to the main result of this section.
\begin{remark}\label{solution_for_nilpotent}
Consider  a pair $(G^{-1},H)\in (GL_{n}(\mathbb C),GL_{n}(\mathbb C))$, and, using Lemma \ref{lemma_true_all_*}, let $X_1^{-1} = I_{n-v} \oplus G^{-1}$, $Y_1 = I_{n-v} \oplus H \in GL_{n}(\mathbb C)$  be such that $$(X_1^{-1}U)E(VY_1) =   \begin{bmatrix}
    I_{n-v} & 0 \\
    0 & G^{-1}\mathcal{E}H
\end{bmatrix}, (X_1^{*}U^{-*})E(VY_1) = \begin{bmatrix}
    \mathcal{R}_{E^{*}Q} & 0 \\
    0 & G^{*}\mathcal{Q}H
\end{bmatrix}.
$$
Since $(G^{-1}\mathcal{E}H)^*G^{*}\mathcal{Q}H = H^*\mathcal{E}^{*}\mathcal{Q}H$,  finding $G^{-1}$ and $H$ that transforms $(\mathcal{E},\mathcal{Q})$ into canonical form, presents an approach for computing the canonical form for $(E,Q)$.
\end{remark}

Currently, we do not know how to construct  $G^{-1},H \in GL_{n}(\mathbb C)$ as in Remark \ref{solution_for_nilpotent} for general pairs $(E,Q)$. 
In the following, we therefore consider the special cases that $E^*Q$ is Hermitian or skew-Hermitian. By Lemma \ref{lemma_true_all_*}, $(E,Q)$ can be transformed the form $(I_k, \mathcal{R}_{E^*Q}) \oplus (\mathcal{E},\mathcal{Q})$ via $^*$-equivalence, where $\mathcal{R}_{E^*Q}$ is the regular part of the congruence form of $E^*Q$ under $^*$-congruence. The conditions $E^*Q$ being Hermitian and skew-Hermitian imply that  $\mathcal{E}^*\mathcal{Q} = 0$. In view of this observation, we now introduce and address the problem 
of finding a canonical form for a pair $(\mathcal{E},\mathcal{Q})$ under $^*$-equivalence via $U \in GL_{n}(\mathbb C)$ and unitary $V$, and $\mathcal{E}^*\mathcal{Q} = 0$. For this we introduce $k\times
(k+1)$ matrices
$$
F_k :=
\begin{bmatrix}
0&1&&\\
&\ddots&\ddots&\\
&&0&1\\
\end{bmatrix}, \qquad
G_k :=
\begin{bmatrix}
1&0&&\\
&\ddots&\ddots&\\
&&1&0\\
\end{bmatrix},
$$ where the non-specified entries are zeros. We also define the direct sum of two pairs of matrices $(A,B)$ and $(C,D)$ as $(A,B) \oplus (C,D) := (A \oplus C, B \oplus D)$.
Then we have the following theorem.
\begin{theorem}\label{nilpotent}
    Let $(\mathcal{E},\mathcal{Q}) \in (M_{n}(\mathbb C),M_{n}(\mathbb C))$ satisfy $\mathcal{E}^*\mathcal{Q} = 0$. Then there exist nonsingular $U$ and unitary $V$ such that
     $$(U\mathcal{E}V,U^{-*}\mathcal{Q}V) = \bigoplus_{i=1}^f (A_i,B_i);$$ 
     where $(A_i,B_i) \in \{(1,0),(0,1),(0,0),(G_1^T,F_1^T) \oplus (G_0,F_0)\}  $ and $f\in \mathbb{Z}_{+}$. This canonical form is determined uniquely up to a permutation of pairs of summands.
\end{theorem}
\begin{proof}
    Since $\mathcal{E}^*\mathcal{Q} = 0$, also $\mathcal{Q}(\mathcal{E}^*\mathcal{Q})\mathcal{E}^* = 0$, i.e., $\mathcal{Q}\mathcal{E}^*$ is nilpotent of nilpotency index at most two.  Since $ \mathcal{E}^*(\mathcal{Q}\mathcal{E}^*) = 0 \text{ and } \mathcal{Q}(\mathcal{E}^*\mathcal{Q}) = 0$,  the length of the longest finite nonzero alternating product of $\mathcal{E}^*$ and $\mathcal{Q}$ can be at most $2$. Since $\mathcal{E}^*\mathcal{Q} = 0$, we have $\mathcal{Q}\mathcal{E}^* \neq 0$ as the longest nonzero finite alternating product with even parity. The longest nonzero finite alternating product is of odd parity, i.e., either $\mathcal{E}^* \neq 0$ or $\mathcal{Q} \neq 0$. We proceed with  the sub-cases of even and odd parity separately.

    {\bf Even Parity:} Let $x,y \in \mathbb{C}^n$ such that $y^*\mathcal{Q}\mathcal{E}^*x \neq 0$.
    Set 
    \[
    Z_1 = \begin{bmatrix}
        \dfrac{x}{||\mathcal{E}^*{x}||} & \dfrac{\mathcal{Q}\mathcal{E}^*x}{||\mathcal{E}^*{x}||}
    \end{bmatrix}_{n\times 2},
   \ y_1 = \begin{bmatrix}
        \dfrac{\mathcal{E}^*x}{||\mathcal{E}^*{x}||}
    \end{bmatrix}_{n\times 1},\ p= \begin{bmatrix}
        y^*\mathcal{Q}\\
    \end{bmatrix}_{1\times n},\ R = \begin{bmatrix}
        y^*\\
        y^*\mathcal{Q}\mathcal{E}^*
    \end{bmatrix}_{2\times n}.
    \]
    We then  prove by contradiction that $\dfrac{x}{||\mathcal{E}^*{x}||}$ and $\dfrac{\mathcal{Q}\mathcal{E}^*x}{||\mathcal{E}^*{x}||}$ are linearly independent. A similar method can then be used to prove $R$ has full row rank. Let $k$ be a scalar such that $\dfrac{x}{||\mathcal{E}^*\Tilde{x}||} + k\dfrac{\mathcal{Q}\mathcal{E}^*x}{||\mathcal{E}^*\Tilde{x}||}= 0$. Multiplying the equation by $\mathcal{E}^*$ on the left and using that $\mathcal{E}^*\mathcal{Q} = 0$, we get $\mathcal{E}^*x = 0$ which is a contradiction. 
    
    We then extend $y_1$ as $\{y_1, \gamma_1,...,\gamma_{n-1}\}$ to an orthonormal set, so that with $Y_2 = \begin{bmatrix}
        \gamma_1 & \dots & \gamma_{n-1}
    \end{bmatrix} \in M_{n,(n-1)}$ 
    then $Y = \begin{bmatrix}
        y_1 & Y_2
    \end{bmatrix}$ is a unitary matrix. Thus, 
    \begin{equation}\label{py_1}
        py_1 \neq 0 \text{ and } pY_2 = 0.
    \end{equation} 
    We set
    $H = \begin{bmatrix}
        1 & 0
    \end{bmatrix}$ and $K = \begin{bmatrix}
        0 & 1
    \end{bmatrix}$. Observe,  that
    \[
    RZ_1 = \dfrac{1}{||\mathcal{E}^*{x}||}\begin{bmatrix}
        y^*x & y^*\mathcal{Q}\mathcal{E}^*x\\
        y^*\mathcal{Q}\mathcal{E}^*x & 0
    \end{bmatrix}
    \]
    is nonsingular. By simple  calculations  then 
    \begin{equation}\label{E^*Z_1}
        \mathcal{E}^*Z_1 = y_1H,
    \end{equation}
    \begin{equation}\label{pE^*}
        p\mathcal{E}^* = KR,
    \end{equation}
    \begin{equation}\label{Qy_1}
        \mathcal{Q}y_1 = Z_1K^T,
    \end{equation}
    \begin{equation}\label{RQ}
        R\mathcal{Q} = H^Tp.
    \end{equation}
    
    Suppose that $n > 2$. Then, by Lemma \ref{lemma_nonsing}, there exist a matrix $Z_2$ of compatible size such that $Z = \begin{bmatrix}
        Z_1 & Z_2
    \end{bmatrix}$ is nonsingular and 
    \begin{equation}\label{RZ_2_=_O}
        RZ_2 = 0.
    \end{equation} 
    Partition $C := Y^{-1}\mathcal{E}^*Z_2 = \begin{bmatrix}
        C_1 \\
        A_2^*
    \end{bmatrix}$ conformably such that \begin{equation}\label{E^TZ_2}
    \mathcal{E}^*Z_2 = YC = \begin{bmatrix}
        y_1 & Y_2
    \end{bmatrix}\begin{bmatrix}
        C_1 \\ 
        A_2^*
    \end{bmatrix} = y_1C_1 + Y_2A_2^*.
    \end{equation}
    Using \eqref{py_1}, \eqref{pE^*}, \eqref{RZ_2_=_O} and \eqref{E^TZ_2}, we then get
    $$p\mathcal{E}^*Z_2 = py_1C_1 + (pY_2)A_2^* = py_1C_1 \text{ and } p\mathcal{E}^*Z_2 = KRZ_2 = 0.$$
    Since, $py_1 \neq 0$ by \eqref{py_1}, then $C_1 = 0$. Thus, using also \eqref{E^*Z_1}, we obtain
    $$\mathcal{E}^*Z 
    = \begin{bmatrix}
        \mathcal{E}^*Z_1 & \mathcal{E}^*Z_2 
    \end{bmatrix} = \begin{bmatrix}
        y_1H & Y_2A_2^* 
    \end{bmatrix} = 
    \begin{bmatrix}
        y_1 & Y_2 
    \end{bmatrix}\begin{bmatrix}
        H & \\
         & A_2^* 
    \end{bmatrix},
    $$
    and 
    \begin{equation}\label{Y^{-1}E^*Z}
        Y^{-1}\mathcal{E}^*Z = 
        \begin{bmatrix}
        H & \\
         & A_2^* 
    \end{bmatrix}.
    \end{equation}
    Again, partition $D := Z^{-1}\mathcal{Q}Y_2 = \begin{bmatrix}
        D_1 \\  
        B_2 
    \end{bmatrix}$ conformably such that 
    \begin{equation}\label{QY_2}
        \mathcal{Q}Y_2 = ZD = \begin{bmatrix}
        Z_1 & Z_2 
    \end{bmatrix}  \begin{bmatrix}
        D_1 \\  
        B_2 
    \end{bmatrix} = Z_1D_1 + Z_2B_2.   
    \end{equation}
    Then, using \eqref{py_1}, \eqref{RQ}, \eqref{RZ_2_=_O} and \eqref{QY_2}, it follows that 
    \begin{equation}\label{RQY_2}
        R\mathcal{Q}Y_2 = RZ_1D_1 + (RZ_2)B_2 = (RZ_1)D_1 \text{ and } R\mathcal{Q}Y_2 = H^T(pY_2) = 0.
    \end{equation}
    Since $RZ_1$ is nonsingular, then $D_1=0$, and using \eqref{Qy_1} and \eqref{QY_2}, we arrive at the equations
    $$
    \begin{bmatrix}
        \mathcal{Q}y_1 & \mathcal{Q}Y_2 
    \end{bmatrix} = \begin{bmatrix}
        Z_1K^T & Z_2B_2 
    \end{bmatrix} = \begin{bmatrix}
        Z_1 & Z_2 
    \end{bmatrix}\begin{bmatrix}
        K^T & \\
         & B_2 
    \end{bmatrix},$$
    \begin{equation}\label{Z^{-1}QY}
        Z^{-1}\mathcal{Q}Y = \begin{bmatrix}
        K^T & \\
         & B_2 
    \end{bmatrix}.
    \end{equation}
    Setting $Z = U^*$, and rewriting \eqref{Y^{-1}E^*Z} and \eqref{Z^{-1}QY}, we then have 
    $$Y^{-1}\mathcal{E}^*U^* = \begin{bmatrix}
        H & \\
        & A_2^*
    \end{bmatrix} \text{, }
    U^{-*}\mathcal{Q}Y = \begin{bmatrix}
        K^T & \\
        & B_2
    \end{bmatrix},
    $$
    or
    $$ U\mathcal{E}Y = \begin{bmatrix}
        H^T & \\
        & A_2
    \end{bmatrix} \text{, }
    U^{-*}\mathcal{Q}Y = \begin{bmatrix}
        K^T & \\
        & B_2
    \end{bmatrix},$$ 
    where $H = \begin{bmatrix}
        1 & 0
    \end{bmatrix}$ and $K = \begin{bmatrix}
        0 & 1
    \end{bmatrix}$. Note that, $Y^*\mathcal{E}^*\mathcal{Q}Y = \begin{bmatrix}
        0 & \\
         & A_2^*B_2
    \end{bmatrix} = 0$, i.e., $A_2^*B_2 = 0$.

    If $n = 2$ then $Z_2$ is absent, i.e. $Z = Z_1$ and using \eqref{E^*Z_1}, we get
    $$
    \mathcal{E}^*Z = \mathcal{E}^*Z_1 = y_1H = \begin{bmatrix}
        y_1 & Y_2
    \end{bmatrix}\begin{bmatrix}
        H \\ 
        0
    \end{bmatrix}.
    $$
    Then, for $n=2$, \eqref{Y^{-1}E^*Z} is analogous to
    \begin{equation}\label{(REC)Y^{-1}E^*Z}
        Y^{-1}\mathcal{E}^*Z = \begin{bmatrix}
        H \\ 
        0
    \end{bmatrix}, \text{or, } Z^*\mathcal{E}Y = \begin{bmatrix}
        H^T & 0\\
    \end{bmatrix}.
    \end{equation}
 By \eqref{RQY_2} then 
 $R\mathcal{Q}Y_2 = H^T(pY_2) = 0$, but since $R$ is nonsingular, it follows that $ \mathcal{Q}Y_2 = 0$. Using \eqref{Qy_1} it follows that 
 $$\mathcal{Q}Y = \mathcal{Q}\begin{bmatrix}
        y_1 & Y_2 
    \end{bmatrix} = \begin{bmatrix}
        Z_1K^T & 0 
    \end{bmatrix} = Z_1\begin{bmatrix}
        K^T & 0 
    \end{bmatrix} = Z\begin{bmatrix}
        K^T & 0 
    \end{bmatrix}.$$ Therefore, for $n=2$, \eqref{Z^{-1}QY} is analogous to
    \begin{equation}\label{(REC)Z^{-1}QY}
        Z^{-1}\mathcal{Q}Y = \begin{bmatrix}
        K^T & 0 
    \end{bmatrix}.
    \end{equation} 
   {\bf Odd parity:} As explained previously, the longest nonzero finite alternating product of $\mathcal{E}^*$ and $\mathcal{Q}$ with odd parity is either $\mathcal{E}^* \neq 0$ or $\mathcal{Q} \neq 0$. Note, that in this case $\mathcal{Q}\mathcal{E}^* = \mathcal{E}^*\mathcal{Q} = 0$. First, we work with $n >1$.\\
    {\bf The sub-case $\mathcal{E}^* \neq 0$:} Let $y, {x} \in \mathbb{C}^n$ such that $y^*\mathcal{E}^*{x} \neq 0$. Let $z_1 = \dfrac{{x}}{||\mathcal{E}^*{x}||}$ and $p = y^*$.
    Construct 
   \begin{equation}\label{equation1}
        r = p\mathcal{E}^*
    \text{ and }
        y_1 = \mathcal{E}^*z_1.
    \end{equation}
    Note that $py_1 = rz_1 = y^*\mathcal{E}^*z_1 \neq 0$. Then, the following equations hold:
    \begin{equation}\label{equation2}
        \mathcal{Q}y_1 = \mathcal{Q}\mathcal{E}^*z_1 =  z_1 \times 0
    \text{ and }
        r\mathcal{Q} = p\mathcal{E}^*\mathcal{Q} = 0 \times p.
    \end{equation}
    Construct $Y_2$ and $Z_2$ as that in the previous case of even parity, with $Y = \begin{bmatrix}
        y_1 & Y_2
    \end{bmatrix}$  unitary and $Z = \begin{bmatrix}
        z_1 & Z_2
    \end{bmatrix}$  nonsingular such that
   \begin{equation}\label{equation3}
        pY_2 = 0 \text{ and } rZ_2 = 0
    \end{equation}
    Using \eqref{equation1}-\eqref{equation3}, similar arguments as in the previous case, and setting $Z = U^*$, we get the identity 
    $$U\mathcal{E}Y = \begin{bmatrix}
        1 & \\
        & A_2
    \end{bmatrix} \text{, }
    U^{-*}\mathcal{Q}Y = \begin{bmatrix}
        0 & \\
        & B_2
    \end{bmatrix} \text{ , } A_2^*B_2 \text{ is zero} .$$
    {\bf The sub-case $\mathcal{Q} \neq 0$:} We \mbox{\rm inter}change the roles of $\mathcal{E}^*$ and $\mathcal{Q}$ of our analysis in the sub-case of $\mathcal{E}^* \neq 0$. We do so by assuming that there exist complex vectors $y$, ${x}$ such that $y^*\mathcal{Q}{x} \neq 0$, $z_1 = \dfrac{x}{||{x}||}$ and $y_1 = \mathcal{Q}z_1$. Set $r = y^*$ and $p = r\mathcal{Q}$ and note that $pz_1 = ry_1 = y^*\mathcal{Q}z \neq 0$. Then,
    \begin{equation}\label{equation4}
        \mathcal{E}^*y_1 = 
        \begin{bmatrix}
        z_1 \\ 0\end{bmatrix} \text{ and } p\mathcal{E}^* = 0
    \end{equation}
    Using a similar construction of $Z = \begin{bmatrix}
        z_1 & Z_2
    \end{bmatrix}$ and $Y = \begin{bmatrix}
        y_1 & Y_2
    \end{bmatrix}$ as in the previous case (except that we construct a unitary $Z$ and a nonsingular $Y$ such that 
    \begin{equation}\label{equation5}
        rY_2 = 0 \text{ and } pZ_2 = 0.
    \end{equation}
    Using equations \eqref{equation4} and \eqref{equation5} along with similar arguments as in the previous cases, we get,
    $$Y^{-1}\mathcal{Q}Z = \begin{bmatrix}
        1 & \\
        & B_2
    \end{bmatrix} \text{, }
    Z^{-1}\mathcal{E}^*Y = \begin{bmatrix}
        0 & \\
        & A_2^*
    \end{bmatrix},$$
    $$\text{or, } Y^{-1}\mathcal{Q}Z = \begin{bmatrix}
        1 & \\
        & B_2
    \end{bmatrix} \text{, }
    Y^{*}\mathcal{E}Z = \begin{bmatrix}
        0 & \\
        & A_2
    \end{bmatrix}.$$
    Renaming $Y^{-1} = U$, we finally have, 
    $$U\mathcal{Q}Z = \begin{bmatrix}
        1 & \\
        & B_2 
    \end{bmatrix} \text{, }
    U^{-*}\mathcal{E}Z = \begin{bmatrix}
        0 & \\
        & A_2
    \end{bmatrix} \text{, } A_2^*B_2 \text{ is zero} .$$ 
    
In all of the discussed cases, the process continues analogously with $B_2$ and $A_2$. It may happen that at some stage of the reduction process both $B_2$ and $A_2$ are of zero rank, then the pair is $(0,0)$.
    
For $n = 1$, the case that both $\mathcal{E}^* \neq 0$ and $\mathcal{Q} \neq 0$ is not possible.  Thus, the resulting canonical  summand is either $(1,0)$ or $(0,1)$. 

 Since $E$, $Q$ are square  and $H^T$, $K^T$ are $1 \times 2$ rectangular matrices,  for every block $(H^T,K^T)$, there exist a zero column in both $E$ and $Q$. So, if required, we permute the zero column to obtain the canonical blocks $(1,0),(0,1),(0,0)$ and $(G_1^T,F_1^T) \oplus (G_0,F_0)$.

We summarize our approach to obtain the canonical blocks. If one of $\mathcal{E}$ and $\mathcal{Q}$ is nonzero, then we have described a method that  reduces $(\mathcal{E},\mathcal{Q})$ to $(H,K) \oplus (X,Y)$, where $(H,K)$ is a canonical summand and not a direct sum of canonical summands. This $(H,K)$ is obtained based on the parity of the longest nonzero finite alternating product of $\mathcal{E}^*$ and $\mathcal{Q}$. Since the events of $(\mathcal{E},\mathcal{Q})$ being a pair of zero matrices, the longest alternating nonzero finite product of $\mathcal{E}^*$ and $\mathcal{Q}$ having odd, and, respectively, even parity are three mutually exclusive events, the blocks obtained in each of these cases are irreplaceable. If one of $X^*$ and $Y$ is nonzero, the process continues depending on the parity of the longest nonzero finite alternating product of $X^*$ and $Y$, otherwise we are left with a canonical summand $(0,0)$. 
    
\textbf{Uniqueness}: 
We want to prove that for $A^*B = C^*D = 0$, $(A,B)$ is $ ^*\text{equivalent to }$ $ (C,D) \text{ via}$ a nonsingular $U$ and a unitary $V$ if and only if the parity of both the longest alternating nonzero finite product of $A^*$ and $B$ and that of $C^*$ and $D$ is same. First assume that exactly one of $A$ and $B$ is nonzero, then by $^*$-equivalence, exactly one of $C$ and $D$ is nonzero. Without loss of generality, let the nonzero matrices be $B$ and $D$. Clearly, then $A^*B = BA^* = C^*D = DC^* = 0$. The same equalities hold if $A$ and $C$ are nonzero. Thus, the result follows for this case. Now, assume that $A$, $B$, $C$ and $D$ are nonzero. Let the parity of both the longest alternating nonzero finite product of $A^*$, $B$ and that of $C^*$, $D$ be odd, and, respectively, even, respectively, i.e., $BA^* = 0$ and $DC^* \neq 0$. This clearly contradicts the 
$^*$-\text{equivalence} of $(A,B)$ and $(C,D)$ via $U \in GL_{n}(\mathbb C)$ and unitary $V$. The same contradiction argument holds if $BA^* \neq 0$ and $DC^* = 0$. This completes the proof of the assertion.
    
If there exists another reduction process that reduces $(A,B)$ to a different canonical form, say, $(X,Y)$ under $^*$-equivalence via $U \in GL_{n}(\mathbb C)$ and unitary $V$ with canonical summands $(\bigoplus_{i=1}^{k} X_i, \bigoplus_{i=1}^{k}Y_i)$ such that $(X_i,Y_i) \in \{(1,0),(0,1),(0,0),(G_1^T,F_1^T) \oplus (G_0,F_0)\}$. Then, as in the previous paragraph, $YX^* = 0$, whenever $BA^* = 0$ and $YX^* \neq 0$ whenever $BA^* \neq 0$. It is straightforward to check similar results if one of the matrices in the pair $(A,B)$ is zero. This leads to a contradiction to our assumption; it can be concluded that the canonical blocks are uniquely determined up to permutation of the summands.
\end{proof}

In the following remark, we address the problem of finding a canonical form under the condition that the columns of $\begin{bmatrix}
        \mathcal{E}\\\mathcal{Q}
    \end{bmatrix}$ form a Lagrangian subspace, i.e., a subspace ${S} \subset \mathbb{C}^{2n}$ such that $s^T\begin{bmatrix}
        0 & I_n\\
        -I_n & 0
    \end{bmatrix}t = 0$, for all $s,t \in {S}$.
    
\begin{remark}\label{lang}
    From the proof  of Theorem~\ref{nilpotent}, it is clear that under the condition that
    \[
    \mathcal L=\begin{bmatrix}
        \mathcal{E}\\\mathcal{Q}
    \end{bmatrix}
    \]
    is of full column rank, the canonical blocks are reduced to $(1,0),(0,1)$. Thus, if the columns of $\mathcal L$
    %
    span a Lagrangian subspace, then $(1,0),(0,1)$ are the only blocks in the canonical form under $^*$-equivalence via a nonsingular $U$ and a unitary $V$. Note, that a similar result has been derived for regular pencils $\lambda E - Q$ in Proposition 9 of \cite{MeMW18} using the CS decomposition \cite{GoVa13}.
\end{remark}

Our next results present the canonical forms under $^*$-equivalence for a pair of singular matrices $(E,Q)$ under the assumption that $E^*Q$ are Hermitian or skew-Hermitian.
\begin{theorem}\label{gerenal_ccf*}
    Let $(E,Q) \in (M_{n}(\mathbb C),M_{n}(\mathbb C))$ and suppose that $E^*Q$ is either Hermitian or skew-Hermitian. Then there exist $U,V \in GL_{n}(\mathbb C)$ such that
     $$
     (UEV,U^{-*}QV) = \bigoplus_{i=1}^f (A_i, B_i),
     $$ 
     where $f\in \mathbb{Z}_{+}$ and $(A_i,B_i) \in \{ (M,N), (1,0),(0,1),(0,0),(G_1^T,F_1^T) \oplus (G_0,F_0)\}  $\\
     $(M,N) \in \begin{cases}
         \{(1,1),(1,-1)\}, & \text{ if } E^*Q \text{ is Hermitian, }\\
         \{(1,1)\}, & \text{ if } E^*Q \text{ is Hermitian and positive semi-definite, }\\
         \{(1,i),(1,-i)\}, & \text{ if } E^*Q \text{ is skew-Hermitian. }
     \end{cases}$.

This canonical form can be determined uniquely up to a permutation of pairs of summands. Moreover, if the columns of $\begin{bmatrix}
         E\\
         Q
     \end{bmatrix}$ span a Lagrangian subspace, then the canonical blocks are of type $(1,0),(0,1),(M,N)$.
\end{theorem}
\begin{proof}
     We derive the canonical form for the case that $E^*Q$ is Hermitian. The other cases can be derived in a  similar way. By Theorem \ref{E^*Q is hermitian} and Lemma \ref{lemma_true_all_*}, there exist $U,V \in GL_{n}(\mathbb C)$ such that $$(UEV,U^{-*}QV) = (I_{n-v} \oplus \mathcal{E}, I_x \oplus (-I_y) \oplus \mathcal{Q}),$$
    where $\mathcal{E}^*\mathcal{Q} = 0$. Since $E^*Q$ is Hermitian, then $\mathcal{E}^*\mathcal{Q}$ is Hermitian, i.e., $\mathcal{Q}^*\mathcal{E} = \mathcal{E}^*\mathcal{Q} = 0$. Using Theorem \ref{nilpotent}, Remark \ref{solution_for_nilpotent} and Remark \ref{lang}, the assertion follows.
 
    To show the {\bf uniqueness}, note tat for the matrix $E^*Q$, the cases that $E^*Q$ is nonsingular or zero are mutually exclusive. It is clear by Theorems~\ref{E^*Q is hermitian} and \ref{E^*Q_is_sk-hermi}, that blocks of type $(M,N)$ are the only canonical blocks if $E^*Q$ is nonsingular. If $E^*Q$ is neither nonsingular nor zero then by Remark \ref{solution_for_nilpotent}, Theorems \ref{E^*Q is hermitian} and \ref{nilpotent}, the summands are unique up to permutation, 
     \end{proof}
Using Lemma \ref{lemma_true_for_all}, Theorems \ref{E^TQ is symmetric} and \ref{K,L_can_form} and the methods in the proof  of Theorem \ref{nilpotent}, Remark \ref{lang} and Theorem \ref{gerenal_ccf*}, we have an analogous result for $^T$-equivalence.
\begin{theorem}\label{E^TQ_sym_sing}
    Suppose that $(E,Q) \in (M_{n}(\mathbb C),M_{n}(\mathbb C))$ and $E^TQ$ is either symmetric or skew-symmetric. Then there exist $U,V \in GL_{n}(\mathbb C)$ such that
    $$(UEV,U^{-T}QV) = \bigoplus_{i=1}^f (A_i, B_i),$$ 
     where $f\in \mathbb{Z}_{+}$, $(A_i,B_i) \in  \{(1,0),(0,1),(0,0), (G_1^T,F_1^T) \oplus (G_0,F_0), (M,N)\},  $ and,
     $(M,N) = \begin{cases}
         (1,1), & \text{if $E^TQ$ is symmetric,}\\
         (I_2,H_{2}(-1)), & \text{if $E^TQ$ is skew-symmetric.}
     \end{cases}$ The above canonical form can be determined uniquely up to permutation of pair of summands.  Moreover, if columns of $\begin{bmatrix}
         E\\
         Q
     \end{bmatrix}$ span a Lagrangian subspace, then the canonical blocks are of type $(1,0)$ ,$(0,1)$, or $(1,1)$.
\end{theorem} 
\section{Codimension}
In this section, we compute the codimensions of orbits of a pair of matrices  under the action of $^{T}$-equivalence and $^{*}$-equivalence. Following \cite{DeDo11_1,DeDo11,DmKS13,DmKS14}, we do it by computing the dimensions of the solution spaces of the associated systems of matrix equations. We only consider pairs where one of the matrices is nonsingular.
\subsection{Orbit under $^{T}$-equivalence}
Consider the action of $^{T}$-equivalence, i.e., $(GL_{n}(\mathbb C),GL_{n}(\mathbb C))$ acts on $(M_{n}(\mathbb C)$, $M_{n}(\mathbb C))$, i.e., $U,V$ acts on $(E,Q)$ by $(UEV$, $U^{-T}QV)$. Then the  the orbit under this action:
\begin{equation}\label{orbit_T_EQUIV}
    \mathcal{O}(E,Q) = \{ (UEV, U^{-T}QV):U,V \in (GL_{n}(\mathbb C),GL_{n}(\mathbb C)) \}.
\end{equation}
is a smooth manifold over $\mathbb{C}$ in $(M_{n}(\mathbb C)$, $M_{n}(\mathbb C))$   and the next lemma gives the tangent space.
%
\begin{lemma}\label{tangent_space}
   Let $(E,Q) \in (M_{n}(\mathbb C),M_{n}(\mathbb C))$ and let $\mathcal{O}(E,Q)$ in \eqref{orbit_T_EQUIV} be the orbit under the action of $^{T}$-equivalence. Then the tangent space of $\mathcal{O}(E,Q)$ at $(E,Q)$ is 
   $$\mathcal{T}(E,Q) = \{ (XE + EY, -X^TQ + QY) : (X,Y) \in (M_{n}(\mathbb C),M_{n}(\mathbb C)) \}.$$
\end{lemma}
\begin{proof}
    Following the proof of \cite[p. 71]{DeEd95}, let $\epsilon$ be a small scalar and consider the action of $^T$-equivalence by ($I + \epsilon X,I + \epsilon Y)$ on $(E,Q)$. Then, $((I + \epsilon X)E(I + \epsilon Y), (I + \epsilon X)^{-T}Q(I + \epsilon Y))
    = (E + \epsilon(XE + EY) + O(\epsilon^2), Q + \epsilon(-X^TQ + QY) + O(\epsilon^2))$. Hence, the assertion follows.
\end{proof}

The next lemma derives  the codimension of the orbit in \eqref{orbit_T_EQUIV}.
\begin{lemma}\label{cod_T_equiv}
    Let $(E,Q) \in (M_{n}(\mathbb C),M_{n}(\mathbb C))$ and $\mathcal{O}(E,Q)$ in \eqref{orbit_T_EQUIV} be the orbit under the action of $^{T}$-equivalence. Then the codimension of $\mathcal{O}(E,Q)$ is the dimension of solution space of the system:
\begin{equation}\label{sys_eq_tan}
    XE + EY = 0, \text{ } -X^TQ + QY = 0, \text{ } (X,Y) \in (M_{n}(\mathbb C),M_{n}(\mathbb C)).
\end{equation}
\end{lemma}
\begin{proof}
    Since for a pair $(E,Q)\in (M_{n}(\mathbb C)$, $M_{n}(\mathbb C))$, the dimension of the orbit $\mathcal{O}(E,Q)$ is equal to the dimension of the tangent space $\mathcal{T}(E,Q)$ to $\mathcal{O}(E,Q)$ at $(E,Q)$, the codimension of $\mathcal{O}(E,Q)$ is $2n^2$ minus the dimension of $\mathcal{T}(E,Q)$. 
    Let $f$ be the map that sends $(X,Y) \in (M_{n}(\mathbb C),M_{n}(\mathbb C))$ to $(XE + EY, -X^TQ + QY)$ is $\mathcal{T}(E,Q)$. Then the codimension of $\mathcal{O}(E,Q)$ is equal to the dimension of the nullspace of the map $f$ and the assertion follows.
\end{proof}

To determine the dimension of the solution space of \eqref{sys_eq_tan} we use the following lemma.
\begin{lemma}\label{sol_bet_sys}
    Let $(E,Q)$, $(E_1,Q_1) \in (M_{n}(\mathbb C),M_{n}(\mathbb C))$ be such that, for $U,V \in GL_{n}(\mathbb C)$,
    \begin{equation}\label{equi_eq}
    (E_1,Q_1) = (UEV,U^{-T}QV).
\end{equation} Let $(X^{'},Y^{'})\in (M_{n}(\mathbb C)$, $M_{n}(\mathbb C))$ be uch that $X^{'} = UXU^{-1}$ and $Y^{'} = V^{-1}YV$. Then $(X,Y)$ is a solution of the system \eqref{sys_eq_tan} if and only if $(X^{'},Y^{'})$ is a solution of the system 
\begin{equation}\label{sys_eq_tan_new}
    X^{'}E_1 + E_1Y^{'} = 0, \text{ } -X^{'T}Q_1 + Q_1Y^{'} = 0.
\end{equation}
\end{lemma}
\begin{proof}
    Let $(E,Q)$, $(E_1,Q_1)$, $X'$, $Y'$, $X$ and $Y$ be as stated. Multiplying the first equation and second equation of the system \eqref{sys_eq_tan} by $U$ and $U^{-T}$, respectively, on the left and both the equations by $V$ on the right, we get,
$UXEV + UEYV = 0$ 
and $-U^{-T}X^TQV+U^{-T}QYV = 0$ which can be re-written as 
$(UXU^{-1})\cdot (UEV) + (UEV)\cdot (V^{-1}YV) = 0$ 
and $-(U^{-T}X^TU^{T})\cdot (U^{-T}$ $QV) + (U^{-T}QV)\cdot (V^{-1}YV) = 0$.
Then, using \eqref{equi_eq}, we get \eqref{sys_eq_tan_new}. Now it is easy to verify that $(X,Y)$ is a solution of \eqref{sys_eq_tan} if and only if $(X^{'},Y^{'})$ is a solution of \eqref{sys_eq_tan_new}
\end{proof}
By Lemma \ref{sol_bet_sys}, system \eqref{sys_eq_tan} is equivalent to the system \eqref{sys_eq_tan_new}. Hence we can solve the system \eqref{sys_eq_tan} by reducing it to its canonical form under $^{T}$-equivalence.
\begin{theorem}\label{codimension_T}
    Let $(E,Q) \in (M_{n}(\mathbb C),M_{n}(\mathbb C))$ and at least one of $E$ or $Q$ is nonsingular. For $U,V \in GL_{n}(\mathbb C)$, the canonical forms under the transformation $(E,Q) \rightarrow (UEV,U^{-T}QV)$, are:\\
    for nonsingular $E$,
    $$(\mathcal{C}_E, \mathcal{C}_Q) = (I_n,(\bigoplus_{i = 1}^{p} H_{2m_i}(\mu_i)) \oplus (\bigoplus_{j = 1}^{q} \Gamma_{n_j})\oplus (\bigoplus_{k = 1}^{s}J_{r_k}(0))), \text{ and}$$ 
    for nonsingular $Q$,
    $$(\mathcal{C}_E, \mathcal{C}_Q) = ((I_{n-v} \oplus (\bigoplus_{k = 1}^{s}J_{r_k}(0)^{T}),(\bigoplus_{i = 1}^{p} H_{2m_i}(\mu_i)) \oplus (\bigoplus_{j = 1}^{q} \Gamma_{n_j}) \oplus I_v)),$$  where $v = \sum_{k=1}^{s}r_k$, $ 0 \neq \mu \neq (-1)^{n+1}$, $\mu $ is determined up to replacement by $\mu^{-1}$. This canonical form is uniquely determined up to permutation of pairs of summands. 
    
    Define $\mathcal{J}_{r_i} = 
         \begin{cases}
             (I_{r_i},J_{r_i}(0)), \text{ when $E$ is nonsingular,}\\
             (J_{r_i}(0)^T,I_{r_i}), \text{ when $Q$ is nonsingular.}
         \end{cases}$
         
    Then, the codimension of the orbit $\mathcal{O}(E,Q)$ under the action of $^T$-equivalence can be computed as the following sum:
    $$ c_{(E,Q)} = c_{0} + c_{1} + c_{2} + c_{00} + c_{11} + c_{22} + c_{01} + c_{02} + c_{12}.$$
    In this sum, the terms have the following expressions:
    \begin{itemize}
         \item 
         $ c_0 = \sum_{i=1}^{s} \big\lceil\dfrac{r_i}{2}\big\rceil$, $c_1 = \sum_{i=1}^{q} \big\lfloor\dfrac{n_j}{2}\big\rfloor$, $c_2 = \sum_{i=1}^{p} m_i + 2\sum_j \big\lceil\dfrac{m_j}{2}\big\rceil$, where the sum $\sum_j \big\lceil\dfrac{m_j}{2}\big\rceil$ is taken over the blocks $(I_{2m_j},H_{2m_j}((-1)^{m_j}))$ in $(\mathcal{C}_E, \mathcal{C}_Q)$.
         \item $c_{00} = \sum_{i,j=1, i<j}^{s} \mbox{\rm \mbox{\rm inter}}(J_{r_i}(0),J_{r_j}(0)), \text{ where,}\\ 
         \mbox{\rm \mbox{\rm inter}}(J_{r_i}(0),J_{r_j}(0)) =
         \begin{cases}
              r_j, \text{ if } r_j \text{ is even, }\\ 
              r_i, \text{ if } r_j \text{ is odd and } r_i \neq r_j,\\
              r_i + 1, \text{ if } r_j \text{ is odd and } r_i = r_j.
         \end{cases}$
         \item 
         $c_{11} = \sum \min\{n_{i},n_{j}\},$ where the sum runs over all pairs $((I_{n_i},\Gamma_{n_i})$, $(I_{n_j},\Gamma_{n_j}))$, $ i<j,$ in $(\mathcal{C}_E, \mathcal{C}_Q)$ such that $n_i$ and $n_j$ have the parity (both odd or even).
         \item 
         $c_{22} = 2\sum \min\{m_i,m_j\} + 4\sum \min\{m_s,m_t\},$ where the first sum is taken over all pairs $((I_{2m_i},H_{2m_i}(\mu_i)), (I_{2m_j},H_{2m_j}(\mu_j))), i<j$, of blocks in $(\mathcal{C}_E, \mathcal{C}_Q)$ such that either $''\mu_i \neq \mu_j \text{ and } \mu_i\mu_j = 1''$ or $\mu_i = \mu_j \neq \pm 1$; and the second sum is taken over all pairs  of blocks $((I_{2m_s},H_{2m_s}(\mu_s)),(I_{2m_t},H_{2m_t}(\mu_t)))$, $s<t$, in $(\mathcal{C}_E, \mathcal{C}_Q)$ such that $\mu_s = \mu_t = \pm 1$.
         \item 
         $c_{12} = 2\sum \min\{k,l\},$ where the sum is taken over all pairs $((I_k,\Gamma_k),(I_{2l}$, $H_{2l}((-1)^{k+1})))$ of blocks in $(\mathcal{C}_E, \mathcal{C}_Q)$.
         \item $c_{01} = N_{odd}\cdot \sum_{i=1}^{q}n_i,$ where $N_{odd}$ is the number of blocks $\mathcal{J}_{r_i}$ with odd size in $(\mathcal{C}_E, \mathcal{C}_Q)$.
         \item $c_{02} = N_{odd}\cdot \sum_{i=1}^{p}m_i,$ where $N_{odd}$ is the number of blocks $\mathcal{J}_{r_i}$ with odd size in $(\mathcal{C}_E, \mathcal{C}_Q)$.
    \end{itemize}
\end{theorem}
\begin{proof} Note that by Theorem \ref{CCF_*_equi}, the canonical summands of $E^TQ$ under $^T$-congruence are in one-to-one correspondence to those of $(E,Q)$ under 
$^T$-equivalence.  We consider both the cases of $E$ and $Q$ being nonsingular separately.\\ 
\textbf{E is nonsingular:}
In this case, $(\mathcal{C}_E, \mathcal{C}_Q) = (I_{n}, (\bigoplus_{i = 1}^{p} H_{2m_i}(\mu_i)) \oplus (\bigoplus_{j = 1}^{q} \Gamma_{n_j}) \oplus (\bigoplus_{k = 1}^{s}J_{r_k}(0))).$ Substituting $(\mathcal{C}_E, \mathcal{C}_Q)$ in system \eqref{sys_eq_tan}, we get, 
   \begin{equation}\label{sys_eq_redu}
        X + Y = 0, \text{  } -X^T\mathcal{C}_Q + \mathcal{C}_QY = 0, \text{ or equivalently, } Y^T\mathcal{C}_Q + \mathcal{C}_QY = 0.
    \end{equation}
    Because of Lemma \ref{sol_bet_sys} and \ref{cod_T_equiv}, our problem has boiled down to finding the dimension of the solution space of \eqref{sys_eq_redu} which is the codimension of the orbit of $\mathcal{C}_Q$ under the action of $^{T}$-congruence \cite{DeDo11}. Thus, the assertion follows from Theorem~\ref{cod_cong}.\\
\textbf{$Q$ is nonsingular:} Let $\mathcal{C}_Q = 
\begin{bmatrix}
    \mathcal{C}_{Q_1} & 0\\
    0 & I_{v}
\end{bmatrix}, \mathcal{C}_E = 
\begin{bmatrix}
    I_{n-v} & 0\\
    0 & \mathcal{C}_{E_1}
\end{bmatrix},$ where $\mathcal{C}_{Q_1} = (\bigoplus_{i = 1}^{p} H_{2m_i}) \oplus (\bigoplus_{j = 1}^{q} \Gamma_{n_j})$, $\mathcal{C}_{E_1} = \bigoplus_{k = 1}^{s}J_{r_k}(0)^{T}$. Partition matrices X, Y in \eqref{sys_eq_redu} conformably such that \eqref{sys_eq_redu} is equivalent to
\begin{equation*}
    \begin{bmatrix}
        X_{11} & X_{12}\\
        X_{21} & X_{22}
    \end{bmatrix}
    \begin{bmatrix}
        I_{n-v} & 0\\
        0 & \mathcal{C}_{E_1}
    \end{bmatrix} +
    \begin{bmatrix}
        I_{n-v} & 0\\
        0 & \mathcal{C}_{E_1}
    \end{bmatrix}
    \begin{bmatrix}
        Y_{11} & Y_{12}\\
        Y_{21} & Y_{22}
    \end{bmatrix} = 0
\end{equation*}
\begin{equation*}
    -\begin{bmatrix}
        X_{11} & X_{12}\\
        X_{21} & X_{22}
    \end{bmatrix}^T
    \begin{bmatrix}
        \mathcal{C}_{Q_1} & 0\\
        0 & I_{v}
    \end{bmatrix} +
    \begin{bmatrix}
        \mathcal{C}_{Q_1} & 0\\
        0 & I_{v}
    \end{bmatrix}
    \begin{bmatrix}
        Y_{11} & Y_{12}\\
        Y_{21} & Y_{22}
    \end{bmatrix} = 0.
\end{equation*} i.e.,
\begin{equation*}
    \begin{bmatrix}
        X_{11} & X_{12}\\
        X_{21} & X_{22}
    \end{bmatrix}
    \begin{bmatrix}
        I_{n-v} & 0\\
        0 & \mathcal{C}_{E_1}
    \end{bmatrix} +
    \begin{bmatrix}
        I_{n-v} & 0\\
        0 & \mathcal{C}_{E_1}
    \end{bmatrix}
    \begin{bmatrix}
        Y_{11} & Y_{12}\\
        Y_{21} & Y_{22}
    \end{bmatrix} = 0
\end{equation*}
\begin{equation*}
    -\begin{bmatrix}
        X_{11}^T & X_{21}^T\\
        X_{12}^T & X_{22}^T
    \end{bmatrix}
    \begin{bmatrix}
        \mathcal{C}_{Q_1} & 0\\
        0 & I_{v}
    \end{bmatrix} +
    \begin{bmatrix}
        \mathcal{C}_{Q_1} & 0\\
        0 & I_{v}
    \end{bmatrix}
    \begin{bmatrix}
        Y_{11} & Y_{12}\\
        Y_{21} & Y_{22}
    \end{bmatrix} = 0.
\end{equation*}
i.e.,
\begin{equation} \label{eq_13_cod}
    X_{11} + Y_{11} = 0, 
\end{equation} 
\begin{equation} \label{eq_14_cod}
    X_{12}\mathcal{C}_{E_1} + Y_{12} = 0,
    \end{equation}
\begin{equation} \label{eq_15_cod}
    X_{21} + \mathcal{C}_{E_1}Y_{21} = 0,
\end{equation}
\begin{equation} \label{eq_16_cod}
    X_{22}\mathcal{C}_{E_1} + \mathcal{C}_{E_1}Y_{22} = 0,
\end{equation}
\begin{equation} \label{eq_17_cod}
    -X_{11}^T\mathcal{C}_{Q_1} + \mathcal{C}_{Q_1}Y_{11} = 0,
\end{equation}
\begin{equation} \label{eq_18_cod}
    -X_{21}^T + \mathcal{C}_{Q_1}Y_{12} = 0,
\end{equation}
\begin{equation} \label{eq_19_cod}
    -X_{12}^T\mathcal{C}_{Q_1} + Y_{21} = 0,
\end{equation}
\begin{equation} \label{eq_20_cod}
    -X_{22}^T + Y_{22} = 0.
\end{equation}
Rewriting \eqref{eq_18_cod}, \eqref{eq_19_cod}, we get, respectively,
\begin{equation} \label{eq_21_cod}
    X_{21} = Y_{12}^T\mathcal{C}_{Q_1}^{T},
    \end{equation}
\begin{equation} \label{eq_22_cod}
    X_{12} = \mathcal{C}_{Q_1}^{-T}Y_{21}^T.
\end{equation}
Using \eqref{eq_22_cod}, \eqref{eq_21_cod}, \eqref{eq_13_cod} and \eqref{eq_20_cod} in \eqref{eq_14_cod}, \eqref{eq_15_cod}, \eqref{eq_17_cod} and \eqref{eq_16_cod}, respectively, we get the following new system of equations, for which the dimension of the solution space is our desired codimension.
\begin{equation} \label{eq_23_cod}
    \mathcal{C}_{Q_1}^{-T}Y_{21}^T\mathcal{C}_{E_1} + Y_{12} = 0, \text{ i.e., } Y_{21}^T\mathcal{C}_{E_1} + \mathcal{C}_{Q_1}^{T}Y_{12} = 0, \text{ i.e., }
    Y_{12}^T\mathcal{C}_{Q_1} + \mathcal{C}_{E_1}^TY_{21} = 0,
\end{equation}
\begin{equation} \label{eq_24_cod}
    Y_{12}^T\mathcal{C}_{Q_1}^{T} + \mathcal{C}_{E_1}Y_{21} = 0, \text{ i.e., }
    Y_{21}^T\mathcal{C}_{E_1}^T + \mathcal{C}_{Q_1}Y_{12} = 0,
\end{equation}
\begin{equation} \label{eq_26_cod}
    X_{11}^T\mathcal{C}_{Q_1} + \mathcal{C}_{Q_1}X_{11} = 0.
\end{equation}
\begin{equation} \label{eq_25_cod}
    X_{22}\mathcal{C}_{E_1} + \mathcal{C}_{E_1}X_{22}^T = 0,
\end{equation}

Since the equations \eqref{eq_26_cod} and \eqref{eq_25_cod} are independent, the dimensions of their solution spaces 
are codimensions of orbits of $\mathcal{C}_{E_1}$ and $\mathcal{C}_{Q_1}$, respectively, under the action of $^{T}$-congruence, i.e., the dimension of the solution space of system of \eqref{eq_26_cod} and \eqref{eq_25_cod} is $c_0 + c_{00} + c_{1} + c_{2} + c_{12}$, by Theorem \ref{cod_cong}.
By definition \ref{inter}, the dimension of solution space of \eqref{eq_23_cod} and \eqref{eq_24_cod} is $\mbox{\rm \mbox{\rm \mbox{\rm inter}}} (\mathcal{C}_{E_1}^T,\mathcal{C}_{Q_1})$, i.e., $ \mbox{\rm inter}(\bigoplus_{k = 1}^{s}J_{r_k}(0),(\bigoplus_{i = 1}^{p} H_{2m_i}(\mu_i)) \oplus (\bigoplus_{j = 1}^{q} \Gamma_{n_j}))$. 

Now, we prove that $\mbox{\rm inter}(\bigoplus_{k = 1}^{s}J_{r_k}(0),(\bigoplus_{i = 1}^{p} H_{2m_i}(\mu_i)) \oplus (\bigoplus_{j = 1}^{q} \Gamma_{n_j})) = \mbox{\rm inter}(\bigoplus_{k = 1}^{s}J_{r_k}(0)$, $\bigoplus_{i = 1}^{p} H_{2m_i}(\mu_i))+\mbox{\rm inter}(\bigoplus_{k = 1}^{s}J_{r_k}(0),\bigoplus_{i = 1}^{p} H_{2m_i}(\mu_i))$. For $k = 1$, it is straightforward from the proof of \cite[Lemma 10, p.60]{DeDo11} that for $X \in GL_{n}(\mathbb C)$, 
$$\mbox{\rm \mbox{\rm inter}}(J_{r_1}(0),X) = 
\begin{cases}
    \text{number of rows of } X,& \text{ if } r_1 \text{ is odd},\\
    0, &\text{ if } r_1 \text{ is even}.
\end{cases}$$ Since \text{number of rows of } $\mathcal{C}_{Q_1}= \sum_{i = 1}^{p} 2m_i \text{ } + \text{ } \sum_{j = 1}^{q} n_j$. Thus, $\mbox{\rm inter}(J_{r_1}(0),\mathcal{C}_{Q_1})=inter(J_{r_1}(0),\bigoplus_{i = 1}^{p} H_{2m_i}(\mu_i))\text{ } + \text{ }\mbox{\rm inter}(J_{r_1}(0),\bigoplus_{j = 1}^{q} \Gamma_{n_j})$. For the scenario of $k=s$ in $\mathcal{C}_{E_1}$, i.e., $\mathcal{C}_{E_1} = \bigoplus_{k = 1}^{s}J_{r_k}(0)$, using \eqref{eq_23_cod} and \eqref{eq_24_cod}, we get that $Y_{21}^T\mathcal{C}_{E_1}^T = \mathcal{C}_{Q_1}\mathcal{C}_{Q_1}^{-T}Y_{21}^T\mathcal{C}_{E_1}$ which leads to following the solution for  $Y_{12}$.

Partition $Y_{21}^T$ conformably such that
$$\begin{bmatrix}
    U_1 & \dots & U_s
\end{bmatrix}\begin{bmatrix}
    J_{r_1}(0) & &\\
    & \ddots & \\
    & & J_{r_s}(0)
\end{bmatrix} = \mathcal{C}_{Q_1}\mathcal{C}_{Q_1}^{-T}
\begin{bmatrix}
    U_1 & \dots & U_s
\end{bmatrix}\begin{bmatrix}
    J_{r_1}(0)^T & &\\
    & \ddots & \\
    & & J_{r_s}(0)^T
\end{bmatrix},$$  $$\text{i.e., }\begin{bmatrix}
        U_1J_{r_1}(0) & \dots &
        U_sJ_{r_s}(0)
\end{bmatrix} = \mathcal{C}_{Q_1}\mathcal{C}_{Q_1}^{-T}
\begin{bmatrix}
        U_1J_{r_1}(0)^T & \dots &
        U_sJ_{r_s}(0)^T
\end{bmatrix},$$ $$\text{or, }
        U_1J_{r_1}(0) = \mathcal{C}_{Q_1}\mathcal{C}_{Q_1}^{-T}
        U_1J_{r_1}(0)^T, \dots,
        U_s J_{r_s}(0) = \mathcal{C}_{Q_1}\mathcal{C}_{Q_1}^{-T}
        U_sJ_{r_s}(0)^T.$$ 
        
        Since the columns of $
        U_1$, \dots, $U_r$ are independent of each other, following the proof of \cite[Lemma 10, p.60]{DeDo11} and assuming that $N_{odd}$ is the total number of odd sized Jordan blocks in $\mathcal{C}_{E_1}$, 
        \hide{if both $r_1$ and $r_2$ are even, $ 
         \begin{bmatrix}
        U_1\\
    U_3 
    \end{bmatrix} = 
    \begin{bmatrix}
        U_2\\
    U_4 
    \end{bmatrix} = 0, i.e., Y_{21} = 0 \text{ that implies } Y_{12} = 0$;
    if $r_1$ is odd and $r_2$ is even,
   $\begin{bmatrix}
        U_2\\
    U_4 
    \end{bmatrix} = 0$ and} 
    $\mbox{\rm \mbox{\rm inter}}(\bigoplus_{k = 1}^{s}J_{r_k}(0),\mathcal{C}_{Q_1}) = N_{odd}(\text{number of rows of } \mathcal{C}_{Q_1}) = N_{odd}(\sum_{i = 1}^{p} 2m_i \text{ } + \text{ } \sum_{j = 1}^{q} n_j) = N_{odd}(\sum_{i = 1}^{p} 2m_i) \text{ } +  N_{odd}(\text{ } \sum_{j = 1}^{q} n_j), $ i.e., $\mbox{\rm \mbox{\rm inter}}(\mathcal{C}_{E_1}^T$, $\mathcal{C}_{Q_1}) = c_{01} + c_{02}$ by Theorem \ref{cod_cong}. 

Hence, the assertion follows.
\end{proof} 

We illustrate Theorem \ref{codimension_T} by the following example of a pair of $6 \times 6$ matrices.
\begin{example}\label{example}
    Let $(E,Q) \in (M_6,M_6)$. The canonical form under the transformation $(E,Q) \rightarrow (UEV,U^{-T}QV)$, for $U,V \in GL_{n}(\mathbb C)$, is $(\mathcal{C}_E, \mathcal{C}_Q) = (I_{4} \oplus J_{2}(0)^{T}, H_{2}(\mu) \oplus \Gamma_{2} \oplus I_2)$, where $\mu = 2$. This is the case of nonsingular $Q$ in Theorem \ref{codimension_T}. Using the same partition for $\mathcal{C}_E, \mathcal{C}_Q, X, Y$ and the same computations, we arrive at the system of equations:
    \begin{equation} \label{eq_1_ex}
    Y_{12}^T(H_{2}(\mu) \oplus \Gamma_{2}) + J_{2}(0)Y_{21} = 0,
\end{equation}
\begin{equation} \label{eq_2_ex}
    Y_{21}^TJ_{2}(0) + (H_{2}(\mu) \oplus \Gamma_{2})Y_{12} = 0,
\end{equation}
\begin{equation} \label{eq_3_ex}
    X_{11}^T(H_{2}(\mu) \oplus \Gamma_{2}) + (H_{2}(\mu) \oplus \Gamma_{2})X_{11} = 0,
\end{equation}
\begin{equation} \label{eq_4_ex}
    X_{22}J_{2}(0)^T + J_{2}(0)^TX_{22}^T = 0.
\end{equation} As in the proof of Theorem \ref{codimension_T}, the dimension of the solution space of \eqref{eq_3_ex} and \eqref{eq_4_ex} is the sum of codimensions of orbits of $H_{2}(\mu) \oplus \Gamma_{2}$ and $J_{2}(0)$ under the action of $^T$-congruence, which by Theorem \ref{cod_cong} is $2+1 = 3$.
Also, according to Definition \ref{inter}, the dimension of the solution space of \eqref{eq_1_ex} and \eqref{eq_2_ex} is $\mbox{\rm \mbox{\rm inter}}(J_{2}^T(0),H_{2}(\mu) \oplus \Gamma_{2})$ which by the proof method of Theorem \ref{codimension_*} is $0$. Thus, the  codimension
 is $3$. 
\end{example}

\subsection{$^{*}$-equivalence for a pair of matrices}
Consider the action of $^{*}$-equivalence, i.e., that $(GL_{n}(\mathbb C)$, $GL_{n}(\mathbb C))$ acts on $(M_{n}(\mathbb C)$, $M_{n}(\mathbb C))$, i.e., $U,V$ acts on $(E,Q)$ by $(UEV, U^{-*}QV)$. Then, the orbit under this action is
\begin{equation}\label{orbit_*_EQUIV}
    \mathcal{O}^{*}(E,Q) = \{ (UEV, U^{-*}QV): (U,V) \in (GL_{n}(\mathbb C),GL_{n}(\mathbb C)) \}.
\end{equation} Since, this above action is not compatible with complex-analyticity, $\mathcal{O}^{*}(E,Q)$ is a smooth real manifold in $(M_{n}(\mathbb C),M_{n}(\mathbb C))$. We omit the proof of the following lemma as it is similar to that of Lemma \ref{tangent_space}. 
\begin{lemma}\label{tangent_space_*}
   Let $(E,Q) \in (M_{n}(\mathbb C),M_{n}(\mathbb C))$ and let $\mathcal{O}^{*}(E,Q)$ in \eqref{orbit_*_EQUIV} be the orbit under the action of $^{*}$-equivalence. Then the tangent space of $\mathcal{O}^{*}(E,Q)$ at $(E,Q)$ is 
   $$\mathcal{T}^{*}(E,Q) = \{(XE + EY, -X^{*}Q + QY) : (X,Y) \in (M_{n}(\mathbb C),M_{n}(\mathbb C))\}.$$
\end{lemma}
The next lemma states the (real) codimension of $\mathcal{O}^{*}(E,Q)$ in \eqref{orbit_*_EQUIV} and the arguments involved in its proof are similar to that of Lemma \ref{cod_T_equiv}.
\begin{lemma}\label{cod_T_equiv1}
    Let $(E,Q) \in (M_{n}(\mathbb C),M_{n}(\mathbb C))$ and $\mathcal{O}^{*}(E,Q)$ in \eqref{orbit_*_EQUIV} be the orbit under the action of $^{*}$-equivalence. Then the (real) codimension of $\mathcal{O}^{*}(E,Q)$ is the dimension of the solution space of the following system:
\begin{equation}\label{*eq-1}
    XE + EY = 0, \text{ } -X^*Q + QY = 0, \text{ } (X,Y) \in (M_{n}(\mathbb C),M_{n}(\mathbb C)).
\end{equation}
\end{lemma}
The following lemma is needed to prove our desired result.
\begin{lemma}\label{sol_bet_sys_*}
    Let $E,Q,E_1,Q_1 \in M_{n}(\mathbb C)$ such that for $U,V \in GL_{n}(\mathbb C)$, $(E_1,Q_1) = (UEV,U^{-*}QV)$. Let $(X^{'},Y^{'})\in (M_{n}(\mathbb C)$, $M_{n}(\mathbb C))$ such that $X^{'} = UXU^{-1}$ and $Y^{'} = V^{-1}YV$. Then $(X,Y)$ is a solution of the system \eqref{*eq-1} if and only if $(X^{'},Y^{'})$ is a solution of the system \begin{equation}\label{*eq-3}
    X^{'}E_1 + E_1Y^{'} = 0, \text{ } -X^{'*}Q_1 + Q_1Y^{'} = 0.
\end{equation}
\end{lemma}
\begin{proof}
    The proof is similar to that of Lemma \ref{sol_bet_sys}, the only difference being that second equation of \eqref{*eq-1} is multiplied by $U^*$ instead of $U^T$.
\end{proof}

Since \eqref{*eq-1} and \eqref{*eq-3} are equivalent,  we can solve the system \eqref{*eq-1} for the canonical blocks of $(E,Q)$ under $^{*}$-equivalence. The next theorem is our main result in this subsection.
\begin{theorem}\label{codimension_*}
    Let $(E,Q) \in (M_{n}(\mathbb C),M_{n}(\mathbb C))$ and at least one of $E$ or $Q$ is nonsingular. For $U,V \in GL_{n}(\mathbb C)$, he canonical form under the transformation $(E,Q) \rightarrow (UEV,U^{-*}QV)$  is $(\overline{\mathcal{C}}_E, \overline{\mathcal{C}}_Q)$, where,\\ for nonsingular $E$,
$$(\overline{\mathcal{C}}_E, \overline{\mathcal{C}}_Q) = (I_n,(\bigoplus_{i = 1}^{p} H_{2m_i}(\mu_i)) \oplus (\bigoplus_{j = 1}^{q}e^{i\theta_{j}}\Delta_{n_j})\oplus (\bigoplus_{k = 1}^{s}J_{r_k}(0)) \text{ and}$$
    for nonsingular $Q$,
$$(\overline{\mathcal{C}}_E, \overline{\mathcal{C}}_Q) = (I_{n-v} \oplus (\bigoplus_{k = 1}^{s}J_{r_k}(0)^{*}),(\bigoplus_{i = 1}^{p} H_{2m_i}(\mu_i)) \oplus (\bigoplus_{j = 1}^{q}e^{i\theta_{j}}\Delta_{n_j})\oplus I_v),$$ where $v = \sum_{k=1}^{s}r_k$, $|\mu_i| > 1, \text{ and } 0\le \theta_j < 2\pi .$
     This canonical form is uniquely determined up to permutation of pair of summands. Define 
     $\mathcal{J}_{r_i} = 
         \begin{cases}
             (I_{r_i},J_{r_i}(0)), \text{ when $E$ is nonsingular},\\
             (J_{r_i}(0)^*,I_{r_i}), \text{ when $Q$ is nonsingular}.
         \end{cases}$ 
         
Then the codimension of $\mathcal{O}^{*}(E,Q)$ in \eqref{orbit_*_EQUIV} under the action of $^*$-equivalence can be computed as the sum:
    $$ \overline{c}_{(E,Q)} = \overline{c}_{0} + \overline{c}_{1} + \overline{c}_{2} + \overline{c}_{00} + \overline{c}_{11} + \overline{c}_{22} + \overline{c}_{01} + \overline{c}_{02},$$ where the components are given by the following expressions: 
     \begin{itemize}
         \item
         $ \overline{c}_0 = \sum_{i=1}^{s} 2\big\lceil\dfrac{r_i}{2}\big\rceil$,
         $ \overline{c}_1 = \sum_{i=1}^{q} n_i$,
         $ \overline{c}_2 = \sum_{i=1}^{p} 2m_i$.
         \item
         $\overline{c}_{00} =
              \sum_{i,j=1, i<j}^{s} \mbox{\rm inter} (J_{r_i}(0),J_{r_j}(0)), \text{ where,}\\
        \mbox{\rm inter}(J_{r_i}(0),J_{r_j}(0)) =
         \begin{cases}
              2r_j, \text{ if } r_j \text{ is even, }\\ 
              2r_i, \text{ if } r_j \text{ is odd and } r_i \neq r_j,\\
              2(r_i + 1), \text{ if } r_j \text{ is odd and } r_i = r_j.
         \end{cases}$
         \item $\overline{c}_{11} = \sum \min\{n_{i},n_{j}\},$ where the sum runs over all pairs of blocks $((I_{n_i},e^{i\theta_{i}}\Delta_{n_i})$, $(I_{n_j},e^{i\theta_{j}}\Delta_{n_j}))$, $ i<j,$ in $(\overline{\mathcal{C}}_E, \overline{\mathcal{C}}_Q)$ such that (a) $n_i$ and $n_j$ have the same parity (both odd or both even) and $|e^{i\theta_{i}}|= \pm |e^{i\theta_{j}}|$, and (b) $n_i$ and $n_j$ have different parity and $|e^{i\theta_{i}}|= \pm |e^{i\theta_{j}}|$.
         \item
         $\overline{c}_{22} = 4\sum \min\{m_i,m_j\},$ where the first sum is taken over all pairs  $((I_{2m_i}$, $H_{2m_i}(\mu_i)),(I_{2m_j},H_{2m_j}(\mu_j))), i<j$, of blocks in $(\overline{\mathcal{C}}_E, \overline{\mathcal{C}}_Q)$ such that such that $\mu_i=\mu_j$.
         \item 
         $\overline{c}_{01} = N_{odd}\cdot \sum_{i=1}^{q}2n_i,$ where $N_{odd}$ is the number of $\mathcal{J}_{r_k}$ blocks with odd size in $(\overline{\mathcal{C}}_E, \overline{\mathcal{C}}_Q)$.
         \item 
         $\overline{c}_{02} = N_{odd}\cdot\sum_{i=1}^{p}4m_i,$ where $N_{odd}$ is the number of $\mathcal{J}_{r_k}$ blocks with odd size in $(\overline{\mathcal{C}}_E, \overline{\mathcal{C}}_Q)$.
     \end{itemize}
\end{theorem}
\begin{proof}
Imitating the proof of Theorem \ref{codimension_T}, and considering conjugate transpose of a matrix, $^{*}$-equivalence, $^{*}$-congruence, Theorem \ref{CCF_*_equi_2} and Definition \ref{inter*} instead of transpose of the matrix, $^{T}$-equivalence, $^{T}$-congruence, Theorem \ref{CCF_*_equi} and Definition \ref{inter}, respectively, the result follows using Theorem \ref{*cong_cod}. 

\end{proof}
\section{Conclusion and future work}

In this work, we have derived canonical forms for pairs of complex matrices $(E, Q)$ under transformations motivated by the structure of linear time-invariant (LTI) dissipative Hamiltonian descriptor systems, in particular Lagrange and Dirac subspaces. These forms provide insight into the algebraic properties of matrix pencils such as $\lambda E - Q$ and $\lambda E - LQ$, which are central to the analysis of energy-based dynamical models. We have also addresses structured cases, where $E^TQ$ and $E^*Q$ are, (skew-)symmetric or (skew-)Hermitian, respectively. Furthermore, we have computed the codimensions of the associated equivalence orbits.

Future work may include extending these canonical forms to broader classes of systems, in particular, general singular systems. Another direction is to explore versal deformations \cite{Dmyt25,DmFS12,DmFS14},  stratifications \cite{Dmyt17,DmKa14} of the orbits, and generic orbits \cite{DeDD24a,DeDD24b} to gain deeper insights into the geometric and topological structure of the corresponding spaces.

\section*{Acknowledgements}
The authors thank Christian Mehl for the valuable discussions.
The work was supported by the Swedish Research Council (VR) under grant 2021-05393. 

\bibliographystyle{plain}
\bibliography{library}

\end{document}